\documentclass[12pt]{mymathart}
\usepackage{mymathmacros}
\usepackage[british]{babel}
\usepackage{amsmath,wasysym,latexsym,bbm}
\usepackage{hyperref}
\usepackage{appendix}
\usepackage{commath}

\usepackage{microtype}
\usepackage{mathtools}

\usepackage{calligra}
\DeclareMathAlphabet{\mathcalligra}{T1}{calligra}{m}{n}

\usepackage{color}

\numberwithin{equation}{section}

\mathtoolsset{showonlyrefs} % show the numbers only for equations which are referred

\newcommand{\lrg}[1]{#1}
\newcommand{\adz}[1]{#1}

\newcommand{\hide}[1]{}

\newcommand{\ph}{\varphi}
\renewcommand{\phi}{\ph}

\newcommand{\cyl}{\operatorname{cyl}}
\newcommand{\eucl}{\operatorname{eucl}}
\newcommand{\hyp}{\operatorname{hyp}}
\newcommand{\mcyl}{\operatorname{oc}}
\newcommand{\sph}{\operatorname{sph}}

\newcommand{\modd}[1]{|{\rm d}#1|}
\newcommand{\Deriv}{{\rm D}}

\newcommand{\prioritize}[1]{\left(#1\right)}
\newcommand{\setof}[1]{\left\{#1\right\}}

\newcommand{\annulus}{A}
\newcommand{\bigo}{O}

\newcommand{\classb}{\mathcal{B}}
\newcommand{\closedinterval}[1]{[#1]}
\newcommand{\closedopeninterval}[1]{[#1[}

\newcommand{\complexnumbers}{\C}
\newcommand{\composedwith}{{\circ}}
\newcommand{\contains}{\supset}
\renewcommand{\diameter}{\operatorname{diam}}

\newcommand{\disk}{D}
\newcommand{\distance}{\operatorname{dist}}
\newcommand{\family}[1]{\left(#1\right)}

\newcommand{\familyproduct}[2]{\overset{#2}{\underset{#1}{\prod}}}
\newcommand{\familysum}[2]{\overset{#2}{\underset{#1}{\sum}}}

\newcommand{\familyunion}[2]{\overset{#2}{\underset{#1}{\bigcup}}}
\newcommand{\hausdorffdimension}{\dim}

\newcommand{\imaginarypart}{\operatorname{Im}}

\newcommand{\integerpart}[1]{\left\lfloor#1\right\rfloor}
\newcommand{\integers}{{\mathbb{Z}}}
\newcommand{\iterated}{}

\newcommand{\logarithmicmeasure}{\operatorname{lm}}
\newcommand{\maximummodulus}{M}
\newcommand{\naturalnumbers}{{\mathbb{N}}}

\newcommand{\nonzerocomplexnumbers}{{\withoutzero{{\complexnumbers}}}}

\newcommand{\numberofelements}{\#}
\newcommand{\openclosedinterval}[1]{]#1]}
\newcommand{\openinterval}[1]{]#1[}

\newcommand{\positiverealnumbers}{{{{\realnumbers_{>0}}}}} % LATEX

\newcommand{\realnumbers}{{\mathbb{R}}}
\newcommand{\realpart}{\operatorname{Re}}
\newcommand{\restricted}{|}

\newcommand{\sequence}[1]{\family{#1}}
\newcommand{\setcomplement}{{\backslash}}
\newcommand{\setintersection}{{\cap}}
\newcommand{\setunion}{{\cup}}

\newcommand{\sqrtmone}{{\squarerootofminusone}}
\newcommand{\squarerootofminusone}{i}
\newcommand{\tendsto}{\to}

\newcommand{\twopii}{2\pi\sqrtmone}

\newcommand{\withoutzero}[1]{#1^{{*}}}

\newcommand{\partitionfunction}{Z}

\theoremstyle{changebreak}
\newtheorem{thm}{Theorem}[section]
\newtheorem{lem}[thm]{Lemma}
\newtheorem{obs}[thm]{Observation}
\newtheorem{prop}[thm]{Proposition}

\newtheorem{cor}[thm]{Corollary}

\theoremstyle{defnbreak}
\newtheorem{defn}[thm]{Definition}

\newtheorem{rem}[thm]{Remark}
\newtheorem{rems}[thm]{Remarks}

\newcommand{\Iterated}[1]{#1}
\newcommand{\Disk}{D}

\newcommand{\CoverFamily}{\mathcal{Q}}
\newcommand{\DimensionParameter}{t}

\newcommand{\MainMapping}{f}

\newcommand{\PoincareFunctionPolynomial}{P}
\newcommand{\PoincareFunctionPolynomialDegree}{d}
\newcommand{\RescaledMainMapping}{{\MainMapping}}
\newcommand{\RescalingParameter}{{\lambda}}

\newcommand{\fatouset}{F}
\newcommand{\filledjuliaset}{K}
\newcommand{\hyperbolicdimension}{\operatorname{dim}_{\operatorname{hyp}}}
\newcommand{\juliaset}{J}
\newcommand{\radialjuliaset}{{\juliaset}_{\operatorname{r}}}

\newcommand{\eventualdimension}{\operatorname{edim}}
\newcommand{\eventualhyperbolicdimension}{\operatorname{edim}_{\hyp}}
\newcommand{\evhypdim}{\eventualhyperbolicdimension}

\newcommand{\vanishingexponent}{\theta}
\newcommand{\fixedpoint}{\xi_0}

\newcommand{\multiplier}{{\rho}}
\newcommand{\poincarefunction}{L}
\newcommand{\poincarefunctioninverse}{{\ell}}
\newcommand{\poincarefunctionlevelset}{E}
\newcommand{\postsingularset}{\operatorname{PS}}
\newcommand{\postsingular}{\postsingularset}

\newcommand{\singularset}{S}

\newcommand{\topologicalpressure}{{\mathcal{P}}}

\newcommand{\poincarelevset}{{\poincarefunctionlevelset}}

\newcommand{\lbAnnidx}{k}
\newcommand{\lbAnnlevel}{n}
\newcommand{\lbAnnmod}{m}
\newcommand{\lbAnnulus}{\mathcal{A}}
\newcommand{\lbAnnulusidx}{\lbAnnidx}

\newcommand{\lbCoverit}{k_0}
\newcommand{\lbFundradius}{{R_{f}}}
\newcommand{\lbIfsit}{\nu}
\newcommand{\lbIfsitact}{p}
\newcommand{\lbIfsittransit}{\nu^+}
\newcommand{\lbLargemod}{M}
\newcommand{\lbLargerann}{{\lbAnnulus'}}
\newcommand{\lbPreimginifs}{\omega}
\newcommand{\lbPtinsec}{y}
\newcommand{\lbRadseq}{r}
\newcommand{\lbSector}{S}
\newcommand{\lbSlit}{W}
\newcommand{\lbSlitAnnulus}{\lbAnnulus^{slit}}
\newcommand{\lbSquare}{Q}
\newcommand{\lbSquareimage}{{\lbSquare'}}
\newcommand{\lbWVcentralidx}{N}
\newcommand{\lbWVerror}{\eps}
\newcommand{\lbWVexceptions}{E}
\newcommand{\lbWVpoint}{\xi}
\newcommand{\lbWVpointseq}{\xi}

\newcommand{\ifsdomain}{D}
\newcommand{\ifsname}{\chi}
\newcommand{\ifsrange}{\ifsdomain}

\newcommand{\partfun}{\partitionfunction}

\newcommand{\cylpartfun}{\partitionfunction^{\cyl}}
\newcommand{\cylnorm}[1]{\left\|#1\right\|_{\cyl}}
\newcommand{\cyldiam}{\diameter_{\cyl}}

\newcommand{\mcylpartfun}{\partitionfunction^{\mcyl}}
\newcommand{\mcylnorm}[1]{\left\|#1\right\|_{\mcyl}}

\newcommand{\euclideanpartitionfunction}{\partitionfunction^{\eucl}}
\newcommand{\euclpartfun}{\euclideanpartitionfunction}

\newcommand{\sphericalnorm}[1]{\left\|#1\right\|_{\sph}}

\newcommand{\intersection}{\setintersection}
\newcommand{\hypdim}{\hyperbolicdimension}
\newcommand{\p}{\prioritize}
\newcommand{\union}{\setunion}

\title[Hyperbolic dimension and Poincar{\'e} functions]{Eventual hyperbolic dimension \\
 of entire functions \\
and Poincar{\'e} functions of polynomials}
\author{Alexandre DeZotti}
\address{Dept. of Mathematical Sciences, University of Liverpool, Liverpool L69 7ZL, UK.
ORCiD: 0000-0002-6463-3304. MR Author ID: 906429.}
\email{adezotti@liverpool.ac.uk}
\author{Lasse Rempe-Gillen}
\address{Dept.\ of Mathematical Sciences, University of Liverpool, Liverpool L69 7ZL, UK. 
{ORCiD: 0000-0001-8032-8580; MR Author ID: 738017.}}
\email{l.rempe@liverpool.ac.uk}
\thanks{The second author was supported by a Philip Leverhulme Prize.}

\begin{document}

\begin{abstract}
  Let $P\colon\C\to\C$ be an entire function.
    A \emph{Poincar\'e function} $L\colon\C\to\C$ of $P$ is the entire extension of a linearising 
    coordinate near a repelling fixed point of $P$. We propose such Poincar\'e functions 
    as a rich and natural class of dynamical systems from the point of view of measurable dynamics, showing
    that the measurable dynamics of $P$ influences that of $L$. More precisely, the \emph{hyperbolic dimension}
    $\hyperbolicdimension(P)$ is a lower bound for $\hyperbolicdimension(L)$.
    
Our results 
allow us to describe a large collection of hyperbolic entire functions having full hyperbolic dimension, and hence
    no natural invariant measures. (The existence of such examples was only
         recently established, using very different and much less direct methods.) 
      We also give a negative answer to a natural question concerning the behaviour of 
        \emph{eventual dimensions} under quasiconformal equivalence. 
\end{abstract}%
\maketitle%

\section{Introduction}
 This article is related to the \emph{measurable dynamics} of 
  transcendental entire functions of one complex variable.
  The goal of measurable dynamics is to understand the statistically typical 
  behaviour of a dynamical system, which 
  means finding natural invariant measures that describe the average 
  behaviour of the system in question. 
  In the case of functions of one real variable, the desired measures will often be  
    absolutely continuous with respect to Lebesgue measure so that they represent ``typical behaviour'' in the usual sense;
   compare e.g.~\cite{Lyubich2000}. For
    \emph{complex} one-dimensional systems, this approach must be modified. Indeed, 
     the locus where the interesting dynamics of a complex polynomial $f$ takes place, known  as the \emph{Julia set} $J(f)$, 
    frequently has zero Lebesgue measure and even Hausdorff dimension less than two. 
    Sullivan \cite{Sullivan1982,Sullivan1983} proposed the solution of finding natural geometric measures
    known  as \emph{conformal  measures} % (Definition~\ref{defn: conformal measure})
    to replace Lebesgue measure, and then
    constructing invariant measures that are absolutely continuous with respect to these.

 In the following decades, a very clear picture of this theory has emerged for rational functions.
    In order to provide meaningful information about the dynamics,
    a conformal measure should be supported on the locus of non-uniform hyperbolicity, called the 
    \emph{radial Julia set} $\radialjuliaset(f)$ (Definition~\ref{defn: radial julia set}).
    \emph{Topological Collet-Eckmann maps} (Definition~\ref{defn:TCE}) provide a large class of functions
    for which $\hausdorffdimension(\radialjuliaset(f))= \hausdorffdimension(J(f))$ \cite{Przytycki1998}, 
    and for which the desired conformal measures and invariant measures exist \cite{PrzytyckiRiveraLetelier2007}.
     (Here  $\hausdorffdimension$ denotes
     Hausdorff dimension.)
     Indeed, if one requires additional mixing properties for the invariant measure,
     then their existence becomes equivalent to the topological Collet-Eckmann
      condition; it is plausible that such maps have full measure in the space of rational functions or polynomials of a given degree. 
     Remarkably, they include all \emph{hyperbolic} functions. 

  More recently, the dynamics of transcendental entire functions $f\colon\C\to\C$ has received considerable attention; significant difficulties arise from the
     non-compactness of the phase space and the nature of the transcendental singularity at $\infty$. Early results concerning the 
     measurable dynamics of such a function were mainly negative and highlighted differences to the rational case, such as the existence of hyperbolic 
     examples whose Julia 
     sets have empty interior and positive measure \cite{McMullen1987} or the fact that 
      the complex exponential map is not recurrent \cite{Lyubich1986,Rees1986} and indeed not
      ergodic \cite{Lyubich1987}. These examples suggested that the measurable theory as developed in the rational case breaks down completely
      for transcendental entire functions.
   However,
      subsequently it was realised that the radial Julia set  of such a function
  often has strictly smaller dimension than the full Julia set~--~a phenomenon unknown in rational dynamics until very recently, see below.
  Stallard~\cite[Theorem~C]{Stallard1999} was 
     the first to notice that this may occur; in a breakthrough 
         Urba\'nski and Zdunik later observed that this holds for every hyperbolic exponential map $f$ \cite{UrbanskiZdunik2003} and, crucially,
        constructed conformal and
         invariant measures supported on $\radialjuliaset(f)$. (See Definition~\ref{def: hyperbolic function} for the definition of hyperbolic
        entire functions.)  This raised the question of whether there might not be 
        a general theory of measurable transcendental dynamics after all,
        \lrg{in analogy to} in the rational case. 
       This problem is not only interesting in its own right, but receives additional relevance through recently announced results
        of Avila and Lyubich \cite{Man-AvilaLyubich2015},
        who gave the first examples of rational maps 
       for which the radial Julia set has smaller Hausdorff dimension than the full Julia set. The similarity between this result and the 
       above-mentioned 
       transcendental phenomena suggests that a good understanding of measurable dynamics for transcendental functions will also lead to insights 
       into the most intricate aspects of the polynomial theory.

   Since the seminal work of Urba{\'n}ski and Zdunik, there has been much work on generalising and extending their results. 
    We refer to the survey
       \cite{KotusUrbanski2008} for background and results, and mention here only 
       work by Mayer and Urba{\'n}ski \cite{MayerUrbanski2008}, 
       who treat a large class of hyperbolic transcendental meromorphic functions satisfying a certain strong regularity condition
       (see below~-- this condition is satisfied, in particular, for many maps given by explicit formulae, 
       such as exponential and trigonometric functions). 
        It is natural to ask whether these results might extend to \emph{all} hyperbolic entire functions, or at least those 
      having finite order of growth.
        It turns out that this is not the case: 
        there is a finite-order hyperbolic entire function $f$ for which 
         $\radialjuliaset(f)$ has Hausdorff dimension $2$ and therefore cannot support a conformal measure
         \cite{RempeGillen2014}. 

   Despite much progress, the area suffered somewhat from a lack of good examples
    beyond the explicit families covered by \cite{MayerUrbanski2008}. 
     (We remark that the function in~\cite{RempeGillen2014} is constructed in a very non-explicit and rather artificial manner.) 
     We propose to address this issue by studying the geometric properties of \emph{Poincar{\'e} functions} of polynomials. 
     In particular,
       we show (Corollary~\ref{cor:linearisers with evhypdim 2}) that there is
       a large collection of hyperbolic entire functions $ f $ of finite order satisfying
       $ \hausdorffdimension \p{ \radialjuliaset (f) } = 2 $. This provides 
       an alternative (and simpler) proof of the main result of \cite{RempeGillen2014}. 
      We also answer a natural question arising from work of Stallard and the second author \cite{RempeStallard2010}
                (Corollary~\ref{cor:noninvariance of evhypdim inside qc classes}).

\subsection*{Poincar\'e functions}
  Let $f$ be an entire function,
   and let $\fixedpoint \in \complexnumbers$ be a fixed point of $f$.
   Suppose that this fixed point is repelling; i.e., its \emph{multiplier} $\multiplier = f'(\fixedpoint )$ has modulus greater than $1$. 
   Then, by a classical theorem of Koenigs (see e.g.~\cite[Section~6.3]{Book-Beardon1991}),
   there exists a conformal map $\poincarefunction$, defined near zero, 
   such that $\poincarefunction(0) = \fixedpoint$ and $\poincarefunction$ satisfies    \begin{equation}\label{eq:schroederequation}
   \poincarefunction \left( \multiplier z \right) = f \left( \poincarefunction (z) \right).
   \end{equation}
    Using the functional equation~\eqref{eq:schroederequation}, this linearising function $\poincarefunction$ extends to 
    an entire function $\poincarefunction\colon \C\to\C$ satisfying~\eqref{eq:schroederequation} everywhere, which 
     is called a \emph{Poincar{\'e} function} for~$f$. 

  The functional relation~\eqref{eq:schroederequation} leads to close connections between the dynamics of $f$ and the function-theoretic properties of
    $\poincarefunction$. (As pointed out by Eremenko and Sodin\cite{EremenkoSodin1990}, this idea can be found already in the work 
    of Julia, Fatou and Latt\`es). Furthermore, the function $\poincarefunction$ can be effectively computed using the functional relation; hence it is
    amenable to computer experiments. As far as we are aware, the idea of studying the \emph{dynamics} of Poincar\'e functions~-- as examples that
    are tangible yet quite different from the usually studied families~-- was first proposed
    by Epstein in the 1990s; see \cite{MihaljevicBrandtPeter2012} for another instance of this approach. 
    As we shall show,
    the measurable dynamics of the function $f$ leaves an imprint on the dynamics of $\poincarefunction$ \emph{near infinity},
    making these maps excellent test cases for the above-mentioned questions.
    (Such a connection was conjectured in \cite[Section~7]{EpsteinRempeGillen2015}.)
    To make the preceding statement precise, we shall change
    our point of view slightly, from non-uniformly to \emph{uniformly} hyperbolic behaviour.

 \subsection*{Hyperbolic sets and (eventual) hyperbolic dimension}
    A \emph{hyperbolic set} $K\subset\C$ of an entire function $f$ is a compact and forward-invariant set on which the function is
     uniformly expanding; i.e., $ \abs{ \p{ f^{ \iterated n } }' (z) } \geq \lambda $
     for some $ n \geq 1 $, some $ \lambda > 1 $, and all $ z \in K $.
     The \emph{hyperbolic dimension} of $ f $ measures the size of 
     such uniformly hyperbolic behaviour in geometric terms:
     \begin {equation}
      \hyperbolicdimension (f)
       \defeq
      \sup \setof{ \hausdorffdimension (K) \colon \text{$ K $ is a hyperbolic set for $ f $} }.
     \end{equation}
     It is known that always $ \hyperbolicdimension (f) = \hausdorffdimension \p{ \radialjuliaset (f) } $ \cite{Rempe2009a}; hence this quantity plays an
     important role in measurable dynamics. 
     As mentioned above, the difficulty in studying transcendental functions (when compared with rational ones) arises from the 
   presence of a transcendental singularity at infinity. The following
    notion measures the 
      ``limiting'' properties of hyperbolic dimension as one approaches this singularity.

\begin{defn}[Eventual hyperbolic dimension]\label{defn:eventualhyperbolicdimension}
  Let $f$ be an entire function. The \emph{eventual hyperbolic dimension} of the mapping
  $f$, denoted by $\eventualhyperbolicdimension (f)$, is
    \[
    \eventualhyperbolicdimension (f) \defeq \lim_{R \tendsto \infty} \sup \left\{
    \hausdorffdimension (K) \colon \text{$K$ is a hyperbolic set for $f$ with $\inf_{z\in K}|z| \geq R$} \right\} . \]
\end{defn}
 We will mainly be interested in the case of entire functions $f$ that belong to the \emph{Eremenko-Lyubich class} $\classb$ of functions 
    having a bounded set of critical and asymptotic values. In this case Bara\'nski, Karpi\'nska and Zdunik \cite{BaranskiKarpinskaZdunik2009} proved that always
    $\hyperbolicdimension(f)>1$; their proof shows also that 
     $\eventualhyperbolicdimension(f)\geq 1$.

 The notion of eventual hyperbolic dimension plays a role~-- albeit implicitly~-- 
  already in the work of
    Urba\'nski and Zdunik on exponential maps. Indeed, a key step in
    their construction of conformal measures is to show that that certain measures 
    do not give
    much mass to points close to the essential singularity. Their arguments imply that,
    for exponential maps,     
    $ \eventualhyperbolicdimension (f) = 1 $. In essence, this is what allows 
    one to disregard behaviour near infinity.     
   More generally, any class $ \classb $ entire function covered by the results of Mayer and Urba\'nski also satisfies $\eventualhyperbolicdimension(f)=1$
    (see Proposition \ref{prop:in mu class evhypdim is 1}). 

 \subsection*{Measurable dynamics of linearisers} 
   The following result relates the eventual hyperbolic dimension of a Poincar\'e function
    to the hyperbolic dimension of the original function.
   
  \begin{thm}[Linearisers with $\hyperbolicdimension <2$]  \label{thm:evhypdim of linearisers}%%%
   Let $ P \colon \C \to \C $ be a polynomial of degree at least $2$ with connected Julia set satisfying the topological Collet-Eckmann condition,
   and let $ \poincarefunction $ be a Poincar\'e function associated to a repelling fixed point of $P$.
    Then 
         \[ 1 \leq \hyperbolicdimension (P) = 
            \eventualhyperbolicdimension (\poincarefunction) 
            < 2. \]

   In particular, if $P$ is not 
     conformally conjugate to a power map or a Chebyshev polynomial\footnote{%
  If 
     $ P $ is conformally conjugate to a power map or a Chebyshev polynomial, 
     then $ \poincarefunction $ agrees with the exponential or the cosine function, respectively, up to 
     pre- and post-composition with affine maps.}, then 
     $ \eventualhyperbolicdimension (\poincarefunction) > 1 $.
     
     If additionally $\poincarefunction$ is \emph{of disjoint type}; that is, 
       hyperbolic with connected Fatou set, then 
    $\eventualhyperbolicdimension(\poincarefunction)\leq \hyperbolicdimension(\poincarefunction)<2$. 
  \end{thm}
  
  Observe that the part on the dimension being greater than $ 1 $
  follows from the known fact that $\hyperbolicdimension(P)>1$ for such polynomials \cite{Zdunik1990}.
    This provides many examples of functions having eventual
     hyperbolic dimension strictly greater than one. 
     Some of these functions do not satisfy the conditions of \cite{MayerUrbanski2008},
     and provide an interesting class of examples for the
     further study of measurable transcendental dynamics.

In general we still have some relation between the eventual hyperbolic dimension of a linearizer and the polynomial.
If $ P \colon \C \to \C $ is a polynomial of degree at least $2$
    and  $\poincarefunction$ is a Poincar\'e function associated to a repelling fixed point of $ P $ 
    then
    \begin{equation} \label{eq: comparison evhypdim linearizer and hypdim polynomial}
    \eventualhyperbolicdimension (\poincarefunction) \geq \hyperbolicdimension(P).
    \end{equation}
    This is a consequence of Theorem \ref{thm:Eventual hyperbolic dimension of linearisers II} below.

 On the other hand, by Shishikura's famous results \cite{Shishikura1998}, there is a residual subset of the boundary of the Mandelbrot set
   where all maps have hyperbolic dimension two. Indeed, by \cite[Section 12]{McMullen2000}, this is true in any non-trivial bifurcation locus
   of polynomials (and rational maps). By linearising such functions, we hence obtain many examples of entire functions having hyperbolic dimension two. 

   \begin{cor}[Linearisers with  $\hyperbolicdimension(\poincarefunction)=2$]\label{cor:linearisers with evhypdim 2}
     There exists a residual subset $R$ of the boundary of the Mandelbrot set with the following property. 
     For any $c\in R$,
     the Poincar\'e function $\poincarefunction$ associated to a repelling periodic point of the polynomial
     $z\mapsto z^2 + c$ belongs to class $\classb$, has finite order and
     $\hyperbolicdimension(L)=\eventualhyperbolicdimension(L)=2$.
   \end{cor}

 While the set $R$ is residual, there are no known \emph{explicit} examples with hyperbolic dimension equal to two. 
 However, using classical
   Wiman-Valiron theory, we show also that one inequality \eqref{eq: comparison evhypdim linearizer and hypdim polynomial}
   holds more generally for linearisers of any entire function 
   (polynomial or transcendental, and regardless of the connectivity of the Julia set): 

   \begin{thm}[Eventual hyperbolic dimension of linearisers]\label{thm:Eventual hyperbolic dimension of linearisers II}
   Let $f\colon\C\to\C$ be a nonlinear entire function, and let $\poincarefunction$ be a Poincar\'e function associated to a repelling fixed point of $f$. 
    Then \[\eventualhyperbolicdimension (\poincarefunction) \geq \hyperbolicdimension(f). \]
   \end{thm}

 In particular, we obtain the following completely explicit example of a hyperbolic entire function having hyperbolic dimension two (albeit of infinite order
    and extremely rapid growth).

  \begin{cor}[Exponential lineariser]\label{cor:exponentiallineariser}
    Consider the function $f\colon \C\to\C; z\mapsto 2\pi i e^z$, and let $L$ be a Poincar\'e function of $f$ at the fixed point $2\pi i$. Then
      $\eventualhyperbolicdimension(L)=2$.

   If furthermore $L$ is normalised such that $|L'(0)| < \frac{1}{20}$, then $L$ is hyperbolic with connected Fatou set. 
  \end{cor}
 \begin{remark}
   Observe that the function $L$ even belongs to the \emph{Speiser class} of transcendental entire functions whose set of critical and asymptotic values
    is finite.
\end{remark}

    In order to prove the above theorems, we introduce a useful quantity,
      called the \emph{vanishing exponent} of a function $f\in\classb$. 
      This exponent always provides an upper bound for the eventual hyperbolic
      dimension of $f$; see Section~\ref{sec: vanishing exponent} and, in particular,
      Lemma~\ref{lem:hypdim criterion} and Lemma~\ref{lem:evhypdim leq vanishing exp}.

  \begin{defn}[Vanishing exponent]\label{defn:vanishingexponent}
   Let $f\in\classb$. Then the \emph{vanishing exponent} $\theta(f)$ is defined as 
    \[ \theta(f) \defeq \inf \setof{ t \geq 0 \colon \limsup_{w\to\infty} \sum_{z\in f^{-1}(w)} \p{ \frac{|w|}{|z| \abs{ f'(z)} } }^t = 0 } \in \closedinterval{ 0, +\infty }. \]
  \end{defn}

\subsection*{Eventual hyperbolic dimension and quasiconformal equivalence}
   Two entire functions $f$ and $g$ are called \emph{quasiconformally equivalent} if there are quasiconformal homeomorphisms 
   $\ph,\psi:\complexnumbers\rightarrow\complexnumbers$ such that
   \begin{equation}\label{eq:equivalence relation}
     \psi\composedwith f=g\composedwith \ph,
   \end{equation}
   and \emph{affinely equivalent} if $\ph$ and $\psi$ can be chosen to be affine. Quasiconformal equivalence classes form the natural
    parameter spaces of transcendental entire functions.

  It was proved in \cite{RempeStallard2010} that, for any function $f\in \classb$, the dimension of the \emph{escaping set}
    (those points converging to infinity under iteration) does not change under affine equivalence, and the question was raised
     \cite[Question~1.7]{RempeStallard2010} whether or not this remains true under quasiconformal equivalence.
     The same article also introduced a notion similar to our concept of eventual hyperbolic dimension: the \emph{eventual dimension}
   \[ \eventualdimension(f) = \inf_{R>0} \hausdorffdimension \{z\in\juliaset(f)\colon |f^{\iterated n}(z)|\geq R\text{ for all $n\geq0$}\}. \]
    This quantity is also invariant under affine equivalence; the
    same argument shows that 
     the eventual hyperbolic dimension from Definition~\ref{defn:eventualhyperbolicdimension} is invariant by affine equivalence for functions in
     $\classb$. In fact, it can be shown that, for $f\in\classb$, always
     $\eventualhyperbolicdimension(f) \leq \eventualdimension(f) = \dim(I(f))$. 

  This raises the natural question of whether the eventual hyperbolic dimension remains constant inside quasiconformal classes. 
    We can deduce from Theorem~\ref{thm:evhypdim of linearisers} that this is not the case.

   \begin{cor}[Eventual hyperbolic dimension may change]%%
    \label{cor:noninvariance of evhypdim inside qc classes}
      There exist functions $f,g\in\classb$ of finite positive order such that $f$ and $g$ are quasiconformally equivalent, but such that 
      \[ 1 < \eventualhyperbolicdimension(f) < \eventualhyperbolicdimension(g). \]

    Furthermore, $f$ is quasiconformally conjugate to $z\mapsto e^z-2$ on a neighbourhood of its Julia set, and likewise for $g$. 
   \end{cor} 

\subsection*{Independent work of Mayer and Urba\'nski}
While this article was being completed, and after our results were first announced,
 Volker Mayer and Mariusz Urba\'nski informed us of a new preprint
   \cite{MayerUrbanski2019} also treating the measurable dynamics and thermodynamic 
   formalism of hyperbolic
   entire functions, including Poincar\'e functions of polynomials. In particular, they give an alternative proof of our
   Theorem~\ref{thm:evhypdim of linearisers}, and show that the final inequality can be replaced by 
   $\eventualhyperbolicdimension(L) < \hyperbolicdimension(L)<2$ if $P$ is hyperbolic
     \cite[Theorems~1.6 and~7.3]{MayerUrbanski2019}. Their 
   results also imply an alternative characterisation of our vanishing exponent;
   see Remark~\ref{rem:vanishing exponent and metric, vanishing exponent and class b}, item \ref{it: vanishing exp is finiteness exp}. 
   
\subsection*{Acknowledgements}
  We thank Dave Sixsmith for interesting discussions and feedback on our manuscript. We are also grateful to Volker Mayer and Mariusz Urba\'nski for making us aware of their recent preprints \cite{Mayer17,MayerUrbanski2019}.

\section{Notation and preliminaries}

As usual, $\naturalnumbers$ denotes the set of non-negative integers and $\C$ denotes the complex plane.
   The (Euclidean) disc of radius $R$ around a point $z\in\C$ is denoted $\Disk (z, R)$. 
   If $f$ is a non-constant, nonlinear entire function, then $\juliaset (f)$ and $\fatouset(f)$ denote its Julia and Fatou sets.   The \emph{singular set} $S(f)$ of $f$ is
 the closure of the set of finite critical and asymptotic values of $f$; compare~\cite[Section~2]{Rempe-GillenSixsmith2017}. The \emph{postsingular set} of $f$ is
     \[ \postsingular(f) \defeq \overline{\bigcup_{s\in S (f)} \{f^n(s)\colon n\geq 0\}}. \]

  An important subset of $ \juliaset (f) $ is provided by its \emph{radial} points 
   (also known as \emph{conical} points)~\cite{Przytycki1999,McMullen2000,Rempe2009a}, as mentioned
   in the introduction. 

\begin{defn}[Radial Julia set]\label{defn: radial julia set}
 Let $ f $ be a rational or transcendental meromorphic function.
 The \emph{radial Julia set} of $ f $, denoted $ \radialjuliaset (f) $,
  consists of all points  $ z \in \juliaset (f) $ for which there are a positive number $r$
 and an increasing sequence $ \sequence{ n_k }_{ k \in \naturalnumbers } $ of natural numbers with the following property: 
 for all $ k \in \naturalnumbers $, the inverse branch 
 of $ f^{ n_k } $ that sends $ f^{ n_k } (z) $ to $ z $
 extends to the spherical disc of radius $ r $ around $ f^{ n_k } (z) $. 
\end{defn}

Let $ f $ be a nonlinear entire function. \emph{Partition functions},
 also known as the \emph{Poincar\'e sequence}
 (see \cite{Przytycki1999}) will play an 
 important role in some of our arguments. They measure, in a certain sense, the 
 size of the preimage of a small disc around a point $w$ under the $n$-th iterate of $f$. 
\begin{defn}[Partition functions]
 Let $ f $ again be rational or transcendental meromorphic.
   Let $ \sigma $ be a conformal metric on $ \C $, let $ t > 0 $ and $ w \in \complexnumbers $.
  Then the \emph{Poincar\'e sequence} of $ f $ is the sequence of \emph{partition functions}
\[
 \partfun^{\sigma}(t,f^n,w)
 \defeq\sum_{z \in f^{-n} (w)} \frac{1}{\norm{ \Deriv f^n (z) }_{\sigma}^t},
 \]
  where $\norm{\Deriv f^n(z)}_{\sigma}$  denotes the norm of the derivative of $f^n$ at $z$. 
\end{defn}

 By a \emph{conformal metric} $\sigma$ we mean here a form that can be written as
  $ \rho_\sigma (z) \modd{z} $ with $ \rho_\sigma (z) > 0 $ for all $ z \in \C $;
  the norm of the derivative of a 
  holomorphic function $\phi$ with respect to $\sigma$ is 
     \[ \norm{\Deriv \phi(z)}_{ \sigma } \defeq \lvert \phi'(z) \rvert \cdot \frac{\rho_\sigma(f(z))}{\rho_\sigma(z)}. \]
    The metric used will be one of the Euclidean or spherical metric or the \emph{cylindrical metric} 
     $ { \modd{z} }/{ |z| } $ (defined on $\C\setminus\{0\}$),
  or its one-sided version
  \begin{equation} \label{eqn:cylindermetric}
    \rho_{\mcyl} (z) = \frac{ 1 }{ \max \setof{ |z|, 1 } }.
  \end{equation}
  We shall write $\euclpartfun$, $\cylpartfun$, $\mcylnorm{\Deriv f(z)}$ etc.\ to indicate which metric is used.  
  For polynomials, the choice of metric in the definition of partition functions is usually irrelevant, as long as the metric is 
   defined in a neighbourhood of the Julia set, which is compact.
   For transcendental entire functions, we shall usually use the (one-sided) cylindrical metric.

\begin{prop}[Continuity of partition  functions] \label{prop: continuity of partition functions}
 Let $ \sigma : \complexnumbers \to \positiverealnumbers $ be the bounded density of a conformal metric on $ \complexnumbers $.
 Assume that
 \begin{equation} \label{eq: condition on the metric for continuity of the partition functions}
  \lim_{ \delta \tendsto 0 } \sup \setof{ \abs{ \frac{ \sigma (z') }{ \sigma (z) } - 1 } \colon z, z' \in \complexnumbers, \abs{ z - z' } \leq \delta } = 0.
 \end{equation}

 Let $ f \in \classb $.
 Let $ w \in \complexnumbers \setcomplement \singularset (f) $ be such that 
 \begin{equation} \label{eq: a condition for the continuity of the partition function}
 \inf \setof{ \abs{ f' (z) } \colon z \in f^{ -1 } (w) } > 0.
 \end{equation}
 Then for all $ t \geq 0 $ the partition function
 \begin{equation}
 \xi \mapsto \partfun^{ \sigma } (t, f, \xi) = \familysum{ \zeta \in f^{ -1 } (\xi) }{ } \norm{ \Deriv f (\zeta) }_{ \sigma }^t
 \end{equation}
 is continuous on a neighborhood of $ w $.
\end{prop}

\begin{rems}
\begin{enumerate}
 \item The condition \eqref{eq: condition on the metric for continuity of the partition functions} is satisfied by the one-sided cylindrical metric.
 \item If $ \abs{ w } $ is large enough then condition \eqref{eq: a condition for the continuity of the partition function} is automatically satisfied; 
   see~\eqref{eqn:eremenkolyubichestimate} below.
\end{enumerate}
\end{rems}

\begin{proof}
 Let
 \begin{equation}
  r \defeq \min \setof{ \distance \p{ w, \singularset (f) }, 1 }  > 0.
 \end{equation}
Then for all $ z \in f^{ -1 } (w) $,
 the map $ f $ has a well defined holomorphic inverse branch $ f^{ -1 }_z $ defined on the disk $ \disk \p{ w, r } $
 such that $ f^{ -1 }_z (w) = z $.
 Let $ u \in \openinterval{ 0, 1 } $.
 It follows from the distortion theorem that there exists $ \kappa = \kappa (u) > 1 $ such that
 $ \kappa (u) \tendsto 1 $ as  $ u \tendsto 0 $
 and for all $ \xi \in \disk \p{ w, u r } $
 \begin{equation} \label{eq: continuity of partition function, distortion estimates on derivative}
  \kappa^{ -1 } \leq \abs{ \frac{ \p{ f^{ -1 }_z }' (\xi) }{ \p{ f^{ -1 }_z }' (w) } } \leq \kappa
 \end{equation}
 and
 \begin{equation}
  \kappa^{ -1 } \frac{ \abs{ \xi - w } }{ r }
  \leq
  \frac{ \abs{  f^{ -1 }_z (\xi) -  f^{ -1 }_z  (w) } }{ \abs{ \p{ f^{ -1 }_z }' (w) } } 
  \leq
  \kappa \frac{ \abs{ \xi - w } }{ r }.
 \end{equation}
 
 Using \eqref{eq: a condition for the continuity of the partition function} and the above we can find $ \eps = \eps (u) > 0 $ small enough so that
 if $ \xi \in \disk \p{ w, \eps } $ then 
 for all $ z \in f^{ -1 } (w) $
 \begin{equation}
  \abs{ \frac{ \sigma \p{ f^{ -1 }_z (\xi) } }{ \sigma (z) } - 1 } \leq u.
 \end{equation}
 Moreover we can also assume that
 \begin{equation}
  \abs{ \frac{ \sigma (\xi) }{ \sigma (w) } - 1 } \leq u. 
 \end{equation}

From the above and \eqref{eq: continuity of partition function, distortion estimates on derivative} it follows that there exists $ C = C (u) > 1 $
such that
 $ C (u) \tendsto 1 $ as $ u \tendsto 0 $
and for all $ \xi \in \disk \p{ w, \eps } $ and all $ z \in f^{ -1 } (w) $
\begin{equation}
 C^{ -1 }
 \leq
 \frac{ \norm{ \Deriv f^{ -1 }_z (\xi)  }_{ \sigma } }{ \norm{ \Deriv f^{ -1 }_z (w)  }_{ \sigma } }
 \leq
 C
\end{equation}
Since $ \abs{ \xi - w } < \frac{ r }{ 2 } $
it follows that any inverse branches of $ f $ defined near $ \xi $ is also well defined at $ w $,
hence $ f^{ -1 } (\xi) = \setof{ f^{ -1 }_z (\xi) : z \in f^{ -1 } (w) } $.
Thus
\begin{equation}
 C^{ -t }
 \leq
 \frac{ \partfun^{ \sigma } (t, f, \xi) }{ \partfun^{ \sigma } (t, f, w) }
 \leq
 C^t.
\end{equation}
Consequently $  \partfun^{ \sigma } (t, f, w) < \infty $ if and only if $ \partfun^{ \sigma } (t, f, \xi) < \infty $.
We assume now that $  \partfun^{ \sigma } (t, f, w) < \infty $.
Since $ C (u)^{ \pm t } \tendsto 1 $ as $ u \tendsto 0 $ it follows that $ \partfun^{ \sigma } (t, f, \xi) \tendsto  \partfun^{ \sigma } (t, f, w) $
as $ \xi \tendsto w $.
\end{proof}

 By definition, for every $t<\hyperbolicdimension f$ (with  $f$ an entire function) there is a hyperbolic set  of dimension 
   at least $t$.
   This implies easily that the partition functions of $ f $ with exponent $t$ grow exponentially. 
   We shall require a sharper, but more technical statement, given by 
     Lemma~\ref{lem:lower exp growth} below.
    It states roughly that, if the definition of the partition function
    is modified to count only preimages near a given point in the Julia set, the growth remains
    exponential.
   As we are not aware of a reference, we include a proof in Appendix \ref{sec:hypdim} for completeness.
 
\begin{lem}[Local growth of partition  functions]
  \label{lem:lower exp growth}
  Let $f$ be a non-constant, non-linear entire function, and let $U$ be an open
  set intersecting the Julia set of $f$. Let $t < \hyperbolicdimension f$.
  
  Then there exists a Jordan domain  $D\subset U$,
 positive constants  $ C$,  $\beta$, $a$ and $b$,
 a sequence $ \p{ I_p }_p $ of positive integers,
 a sequence of finite families of simply connected domains $ \p{ D_i^p }_{ 1 \leq i \leq I_p } $, $ D_i^p \subset D $
 and a sequence of integers $(\nu_p)_{p\geq 1}$ converging to $\infty$
 such that $ \abs{ \nu_p - a p } \leq b $ for all $p$, and such that the following holds.

 For all $ p \geq 1 $ and all $ 1 \leq i \leq I_p $,
 $ f^{ \nu_p } : D_i^p \to D $ is a conformal isomorphism.
For any $ w \in D $
 and $ 1 \leq i \leq I_p$, let
   $ z_i $ denote the unique preimage of $ w $ by $ f^{ \nu_p } $ in $ D_i^p $.
Then
  \begin{equation} \label{pressure-explosion-on-the-ifs}
    \familysum{ i = 1 }{ I_p } \abs{ \p{f^{\iterated \nu_p}}' (z_i) }^{- t}
    \geq
    C e^{\beta \nu_p}.
  \end{equation}
\end{lem}

\subsection*{Pressure and topological Collet-Eckmann maps} 

 If $f$ is a polynomial\footnote{%
  In this article,  we consider only the iteration of polynomials and entire functions, and hence restrict our
    discussion to this case. However,
   anything stated for polynomials in this subsection also holds for rational functions
    with non-empty Fatou set, without modifications.},
   then the \emph{topological pressure} of $f$ is defined as     
  \begin{align}\notag
   \topologicalpressure &\colon (0,\infty)\to \R; \\\label{eq:pressure on riemann sphere}
 \topologicalpressure(\DimensionParameter) &= 
   \limsup_{n\tendsto\infty}
   \frac{1}{n}\log \sum_{z\in f^{-n}(w)} \sphericalnorm{\Deriv f^n(z)}^{-\DimensionParameter} = 
            \limsup_{n\to\infty} \frac{1}{n} \log \partfun^{\sph}(\DimensionParameter, f^n , w).
  \end{align}
  This function is defined and independent of $w$ provided that $w$ does not belong to  
    a certain set $ E \subset \C \setminus \postsingularset (f) $ of zero Hausdorff dimension;
    see \cite{Przytycki1999}.
   There are several other (equivalent) definitions of the pressure function, 
    see \cite{Przytycki1999,Przytycki2004}. A key property of the pressure function is that the
    hyperbolic dimension of $f$ coincides with the smallest zero of $\topologicalpressure$; see 
    \cite{DenkerUrbanski1991d,Przytycki1993,Przytycki1999},
    and also \cite{Przytycki2004,PrzytyckiUrbanski2010}. Furthermore,
    $\topologicalpressure(\DimensionParameter)$ is convex in  $\DimensionParameter$. 

\begin{remark} \label{rem:pressure in class B}
 The formula~\eqref{eq:pressure on riemann sphere} can also be used to define the pressure
 for a large class of transcendental entire or meromorphic functions in
 $\classb$ including the Speiser class, see \cite{BaranskiKarpinskaZdunik2012}. Here 
   $\topologicalpressure(\DimensionParameter)$ may be infinite. Again, 
  the hyperbolic dimension is given by the infimum of
   the set of $\DimensionParameter$ for which the pressure is not positive.
  However, we will only use the pressure functions of polynomials.
 \end{remark}
  
 It is possible to characterise the \emph{topological Collet-Eckmann (TCE)  condition} using the 
  values of the pressure function; this is the most convenient definition of TCE for our purposes. 
  For many other equivalent definitions of the TCE condition, 
 see \cite{PrzytyckiRiveraLetelierSmirnov2003}.
\begin{defn}[Topological Collet-Eckmann  maps] \label{defn:TCE}
 A polynomial satisfies the \emph{topological Collet-Eckmann condition}
  if its pressure $\topologicalpressure(\DimensionParameter)$
  is negative for large values of $\DimensionParameter$.
\end{defn}

Since the pressure function is convex in $\DimensionParameter$, for TCE maps there is 
   a unique zero of $\topologicalpressure$,  which coincides with the hyperbolic dimension.
  We will also use the following result.
\begin{thm}[Hausdorff dimensions of polynomial  Julia  sets]\label{thm:dim of julia of tce}
Let $ P $ be a polynomial with  connected Julia set,  and suppose that 
    $ P $  is not conformally conjugate to a Chebyshev polynomial or a power map. 
   Then
 \[ \hausdorffdimension (\juliaset (P)) > 1. \]
If, moreover, $ P $ satisfies the topological Collet-Eckmann condition, then
 \[ 1 < \hausdorffdimension (\juliaset (P)) < 2. \]
\end{thm}
\begin{proof}
 The first inequality is well-known. Indeed, by \cite[Theorem 2]{Zdunik1990},
  the Hausdorff dimension of the Julia set of $ P $ is greater than the Hausdorff dimension of the equilibrium measure 
  of $ P $.
  The equilibrium measure is given by the harmonic measure on $ \juliaset (P) $, as viewed from 
    $\infty$ \adz{\cite[Section~16]{Brolin1965}}. If $ \juliaset (P) $ is connected, then 
    the basin of infinity is simply-connected, and 
    the dimension of the harmonic measure on the boundary of a simply-connected domain
  is $ 1 $ \cite{Makarov1985}; see also
  \cite[Section~VI.5]{GarnettMarshall2005}.

The strict upper bound for TCE maps follows from \cite[Theorem~4.3]{PrzytyckiRiveraLetelierSmirnov2003}.
\end{proof}

\subsection*{The Eremenko-Lyubich class}
 Recall that the \emph{Eremenko-Lyubich  class} $\classb$ is defined as
    \[ \classb \defeq \{f\colon \C\to\C \text{ transcendental entire}\colon S(f)\text{ is bounded}\}. \]
  Suppose that $f\in\classb$ and $ R > 1 $ is such that $S(f)\subset D(0,R)$. 
   Set $W\defeq \C\setminus\overline{D(0,R/2)}$. Then every component
   $T$ of  $f^{-1}(W)$ is simply-connected and  $f\colon T\to W$ is a universal covering map. 
  These components are called \emph{logarithmic tracts} of $f$ (over  $\infty$).

   Suppose furthermore that $R$ was chosen such that additionally
       $\lvert f(0)\rvert < R$. 
      A well-known estimate due to Eremenko and Lyubich \cite[Lemma~1]{EremenkoLyubich1992} states that 
    \begin{equation}\label{eqn:eremenkolyubichestimate}
       \cylnorm{\Deriv f(z)} \geq \frac{1}{4}\log \left\lvert \frac{f(z)}{R}\right\rvert 
    \end{equation}
   whenever $\lvert f(z) \rvert > R$. In particular,  the cylindrical derivative
    of $\cylnorm{\Deriv f(z)}$ at $z$ tends to infinity as $f(z)\to\infty$.

\subsection*{Hyperbolic and disjoint-type entire functions}
\begin{defn}[Hyperbolic functions]\label{def: hyperbolic function}\label{def: postsingular}
A transcendental entire function $f$ is said to be \emph{hyperbolic} if
   $f\in\classb$ 
 and $S(f)$ is a subset the union of the basins of attraction of
 the attracting periodic cycles of $f$. 
\end{defn}
 Equivalently, a transcendental entire function $f$ is hyperbolic if
 $\postsingular(f)$
 is a compact subset of the Fatou set. Compare~\cite{Rempe-GillenSixsmith2017} for a discussion of and background on the definition of hyperbolicity for transcendental functions. 
   
An entire transcendental function is said to be of \textit{disjoint type}
  if it is hyperbolic with connected Fatou set.
  We remark that any function $ f \in \classb $ can be rescaled into a disjoint-type function; 
   more precisely, $ \lambda f $ is of disjoint type provided that $ \abs{ \lambda } $ is small enough;
   see e.g.~\cite[Example p.392]{BaranskiKarpinska2007}.

\section{Vanishing exponent and eventual hyperbolic dimension}
\label{sec: vanishing exponent}

In this section we investigate the eventual hyperbolic dimension for functions in the 
  Eremenko-Lyubich  class, using the vanishing exponent $\vanishingexponent(f)$ defined 
   in Definition~\ref{defn:vanishingexponent}.
   We begin by connecting this exponent with  the eventual hyperbolic dimension. We then  show
    that both exponents are at least $1$, and equal $1$ for some functions in the Mayer-Urba{\'n}ski class of functions satisfying the 
   balanced growth condition (Proposition~\ref{prop:in mu class evhypdim is 1}).
   Finally we turn to the behaviour of both quantities under affine and quasiconformal equivalence.

\subsection*{Vanishing exponent}
 Note that the series in the 
   vanishing exponent from Definition~\ref{defn:vanishingexponent} is
   precisely the cylindrical partition function
      $\cylpartfun(t,f,w)$. 
   We begin by relating this exponent to the behaviour of the partition functions
    of the disjoint type functions $\lambda f$, as $\lambda$ tends to zero.

\begin{lem}
  \label{lem:hypdim criterion}Let $\MainMapping\in\classb$ and $\DimensionParameter\geq 0$. Then the two following statements are
  equivalent.
  \begin{enumerate}
    \item\label{item:partitionfunctiontozero}
      $\cylpartfun (\DimensionParameter, f, w) \tendsto 0$ as $w \tendsto \infty$,
    
    \item 
    $\cylpartfun (\DimensionParameter, \RescalingParameter \MainMapping, w) \tendsto 0$
      as
      $\RescalingParameter \tendsto 0$ uniformly in $w$, for $|w| > 1$.
  \end{enumerate}
\end{lem}
\begin{proof}
  As the cylindrical metric is invariant under linear rescalings, it follows that
  $\cylpartfun (\DimensionParameter, \RescalingParameter \MainMapping, w)
     =
    \cylpartfun \left( \DimensionParameter, \MainMapping, \frac{w}{\RescalingParameter} \right)$
    for all
  $\RescalingParameter \in \complexnumbers^{\ast}$ and all $w$.
 Hence, the two properties are equivalent.
\end{proof}

\begin{rem}[Vanishing exponents]\label{rem:vanishing exponent and metric, vanishing exponent and class b}
\leavevmode%%\vspace{-\baselineskip}%%%
\begin{enumerate}[(a)]
 \item 
   The \emph{vanishing exponent} is precisely the infimum of the 
   values of $\DimensionParameter$ 
    for which the properties of Lemma \ref{lem:hypdim criterion} hold. 
 \item The assumption $f\in\classb$ is not used in the proof.  However, 
     an entire function  $f\notin\classb$ cannot satisfy the statements in 
     Lemma~\ref{lem:hypdim criterion} for any $\DimensionParameter\in\R$;
     compare~\cite{sixsmithclassB}.
 \item Note that $\cylpartfun(\DimensionParameter,f,w) = \mcylpartfun(\DimensionParameter,f,w)$ for
    all sufficiently large $w$. Indeed, if $w$ is sufficiently large, then
    all preimages of  $w$ have modulus at least $1$, and hence the density of the two metrics agree
    at these preimages. Similarly, for $\lvert w\rvert > 1$ and all sufficiently small $\lambda$, 
    $\cylpartfun(\DimensionParameter,\lambda f,w) = \mcylpartfun(\DimensionParameter,\lambda f,w)$.

     So we may replace the cylindrical metric by the one-sided cylindrical metric in 
      Lemma~\ref{lem:hypdim criterion}, and also in the definition of the vanishing exponent.
 \item \label{it: vanishing exp is finiteness exp}
 In \cite[Theorem 4.1]{MayerUrbanski2019}, the authors prove the following important
    statement: if $ \cylpartfun (\tilde{t}, f, w) < \infty $ for sufficiently large $ w $, 
    then~\ref{item:partitionfunctiontozero} holds for $ t > \tilde{t} $. Hence
    $ \theta (f) $ is also the critical exponent for finiteness of the partition function. 
    
      Moreover, the authors introduce another exponent, $\Theta(f)$, related to 
        an ``integral mean spectrum''. For certain functions~-- those which have
        ``negative spectrum'' in the sense of \cite{MayerUrbanski2019}~-- this exponent agrees
        with our vanishing exponent $\theta(f)$. 
 \end{enumerate}
\end{rem}

\begin{lem}[Vanishing exponent and eventual hyperbolic dimension]%%
  \label{lem:hyp dim eventually bounded by}\label{lem:above ev hyp dim}
  \label{lem:evhypdim leq vanishing exp}
  Let $f\in\classb$.
  Then
  \[ \limsup_{\RescalingParameter \tendsto 0} \hyperbolicdimension
     (\RescalingParameter \MainMapping) \leq \vanishingexponent (f) \qquad\text{and}\qquad
    \eventualhyperbolicdimension (f) \leq \vanishingexponent (f). \]
\end{lem}

\begin{proof}
  \lrg{Let $\DimensionParameter > \vanishingexponent (f)$. We shall show that,} if
  $\lvert \RescalingParameter\rvert$ is small enough, the hyperbolic dimension of
  $\RescaledMainMapping_{\RescalingParameter} \defeq \RescalingParameter \MainMapping$
  is smaller than $\DimensionParameter$.

  If $\lvert \lambda\rvert$ is sufficiently small, then  
  $\cylpartfun (\DimensionParameter, f_{\lambda}, w) \leq 1/2$ whenever 
  $|w|\geq 1$.
   Furthermore, again for 
  $\RescalingParameter$ small enough, 
  the singular set of $\RescaledMainMapping_{\RescalingParameter}$ is
   a subset of the disc $\Disk\left(0, \frac{1}{2}\right)$
  and
  $\RescaledMainMapping_{\RescalingParameter}(\overline{\Disk(0, 1)})
   \subset
   \Disk\left(0, \frac{1}{2}\right)$.
  As a consequence, the forward orbits of the elements of the singular set  of
  $\RescaledMainMapping_{\RescalingParameter}$
  belong to
  $\Disk\left(0, \frac{1}{2}\right)$
  and the Julia set lies in the complement of the closed unit disc.
  
  Fix such a $\RescalingParameter$ and let $X$ be a hyperbolic set for
  $ f_{ \lambda } $, say $X\subset D(0,R)$ with $ R > 1 $.   
   Cover \( X \) by a finite family $\CoverFamily$ of closed topological discs
  $Q \subset \overline{D(0,R)}\setminus  D(0,1)$. 
   For such $Q$ and $n\geq 0$, let $\mathcal{U}_n(Q)$ denote the set of connected components of
   $f_{\lambda}^{-n}(Q)$ intersecting $X$. Then $\mathcal{U}_n\defeq \bigcup_{Q\in\CoverFamily} \mathcal{U}_n(Q)$
    is a cover of $X$ for all $n\geq 0$. We shall use these covers to estimate the Hausdorff dimension of $X$
   from above.

    For each 
     $Q\in\CoverFamily$, 
    there is a slightly larger topological disc $\tilde{Q} \contains Q$ disjoint from 
    $D(0,1/2)$, and thus not intersecting $\postsingularset(f)$. Hence, if
     $n\geq 1$ and $U$ is a connected component of $f_{\lambda}^{-n}(Q)$, then $f_{\lambda}^n$ is univalent
     on the component $\tilde{U}$ of $f_{\lambda}^{-n}(\tilde{Q})$ containing $U$. Observe that 
     $\tilde{U}$ does not intersect $D(0,1)$ by choice of $\lambda$. 
     Applying 
     Koebe's distortion theorem to
      a branch $\phi$ 
      of $\log\bigl( (f_{\lambda}^n|_{\tilde{Q}})^{-1}\bigr)$, there is a constant $C_0>1$ (depending only on $\CoverFamily$) 
      such that   
   \[
     \frac{ \cylnorm{ \Deriv f_{\lambda}^n (\zeta) } }{ \cylnorm{ \Deriv f_{\lambda}^n (z) }}
     = 
        \frac{\lvert \phi'(f_{\lambda}^n(z)) \rvert}{\lvert \phi'(f_{\lambda}^n(\zeta))\rvert} \cdot
         \frac{\lvert f_{\lambda}^n(z)\rvert}{\lvert f_{\lambda}^n(\zeta)\rvert}
     \leq 
         R \cdot \frac{ \abs{ \phi'(f_{\lambda}^n(z)) } }{ \abs{ \phi'(f_{\lambda}^n(\zeta))} }
     \leq
     C_0
   \]
     for all $z,\zeta \in U$.  In particular, 
     there is a constant $C_1$ depending only on $\CoverFamily$ such that
    \[     \cyldiam (U)
        \leq
   \frac
     {C_1}
     { \cylnorm{
              \Deriv  \RescaledMainMapping_{\RescalingParameter}^{\Iterated{n}} (z)
               }
     }. \]

Let $Q\in\CoverFamily$ and fix $\zeta_Q\in Q$. For every $n\geq 1$ and all
   $U\in \mathcal{U}_n(Q)$, let $\zeta_U$ be the unique element of  $f_{\lambda}^{-n}(\zeta)\cap U$. If
   $V_1,\dots, V_m$ are the components of $f_{\lambda}^{-1}(U)$ that intersect $X$, then 
    \begin{align*} \sum_{i=1}^m  
          \cylnorm{\Deriv \RescaledMainMapping_{\RescalingParameter}^{\Iterated{n+1}} (\zeta_{V_i})}^{-t} &=
          \cylnorm{\Deriv \RescaledMainMapping_{\RescalingParameter}^{\Iterated{n}} (\zeta_U)}^{-t} \cdot
          \sum_{i=1}^m      
           \cylnorm{\Deriv \RescaledMainMapping_{\RescalingParameter}
                    (\zeta_{V_i})}^{-t} \\ &\leq
           \cylnorm{\Deriv \RescaledMainMapping_{\RescalingParameter}^{\Iterated{n}} (\zeta_U)}^{-t} \cdot
             \cylpartfun (t, f_{\lambda}, \zeta_U ) \leq 
           \frac{1}{2}\cdot \cylnorm{\Deriv \RescaledMainMapping_{\RescalingParameter}^{\Iterated{n}} (\zeta_U)}^{-t}.
          \end{align*}
  By induction,
    \[ 
       \sum_{U\in \mathcal{U}_n(Q)} 
         \cylnorm{\Deriv \RescaledMainMapping_{\RescalingParameter}^{\Iterated{n}} (\zeta_U)}^{-t} \leq 
         2^{-n}. \]
   Hence 
    \begin{align*} \sum_{U\in \mathcal{U}_n} \cyldiam(U)^t  &=
       \sum_{Q\in \CoverFamily}  \sum_{U\in \mathcal{U}_n(Q)} \cyldiam(U)^t \\ &\leq 
        C_1^t\cdot  \sum_{Q\in \CoverFamily} 
       \sum_{U\in \mathcal{U}_n(Q)} 
         \cylnorm{\Deriv \RescaledMainMapping_{\RescalingParameter}^{\Iterated{n}} (\zeta_U)}^{-t} \leq 
         C_1^t\cdot \# \CoverFamily \cdot 2^{-n} \to 0 \end{align*}
     as $n\to \infty$. 
    Thus $ \hausdorffdimension (X) \leq  \DimensionParameter $, and the first claim of the lemma is proved.

The proof of the second claim is very similar.
  Choose $\DimensionParameter>\vanishingexponent(f)$ and let $R_0 > 0$
 large enough so that $\cylpartfun (\DimensionParameter, f, w)\leq 1/2$
 for $|w|>R_0$.

Let $X$ be a hyperbolic set of $f$ which lies outside the disc of radius $R_0+1$. 
We now continue as above, and cover $X$ by a finite union $\CoverFamily$ of closed topological discs $Q$,
  for each of which there is a larger disc $\tilde{Q}$ that does 
  not intersect $D(0,R_0)$ and is contained in a neighbourhood of $X$ on which
  the map $f$ is expanding. The latter assumption  ensures that, while $\tilde{Q}$ may intersect the
  postsingular set, any component of
   $f^{-n}(\tilde{Q})$ that intersects $X$ lies outside $D(0,R_0)$ and is 
   mapped univalently to $\tilde{Q}$. The remainder of
  the proof proceeds as above,  and we conclude that
 $\hausdorffdimension X\leq \DimensionParameter$. 
  This proves the  second claim.
\end{proof}

\begin{cor}[Non-maximal hyperbolic dimension] \label{cor:finitess exponent and hypdim}
 If $\vanishingexponent(f) < 2$ and if $f$ is of disjoint type,
 then $\hyperbolicdimension(f)<2$. 
\end{cor}
\begin{proof}
By Lemma~\ref{lem:hyp dim eventually bounded by}, there is $\RescalingParameter<1$ such that
$\hyperbolicdimension(\RescaledMainMapping_\RescalingParameter)<2$.
Since both $f$ and $\RescaledMainMapping_\RescalingParameter$ are of disjoint type,
 it follows from Theorem 3.1 of \cite{Rempe2009}
 that these maps are quasiconformally conjugate on neighborhoods of their respective Julia sets.

Let $\DimensionParameter_0<2$ be an upper bound for the dimension of the hyperbolic sets of
 $\RescaledMainMapping_\RescalingParameter$. The hyperbolic sets of $f$ are
 the images of the hyperbolic sets of $f_{\lambda}$ under the above-mentioned quasiconformal conjugacy $\psi$.  
 By \cite[Theorem 12]{GehringVaisala1973},
 the dimensions of these images 
 are bounded by a constant $\DimensionParameter_1<2$
 depending only on $\DimensionParameter_0$ and the quasiconstant of $\psi$.
\end{proof}

\begin{remark}
  Suppose that $f$ is a disjoint-type entire function and $\vanishingexponent(f)<2$. 
    It is mentioned in \cite[Section~8]{MayerUrbanski2019}
    that various results from the theory of thermodynamic formalism,
    as developed in \cite{MayerUrbanski2008,MayerUrbanski2010} will carry over for $t>\vanishingexponent(f)$.
    In particular, if 
    $t_0\defeq \hyperbolicdimension(f) > 2$, then there is a probability
     measure that is conformal (with respect to the cylinder metric) and supported on
     the radial Julia set. In particular, it follows that $t_0<2$, giving an alternative
     proof of Corollary~\ref{cor:finitess exponent and hypdim}. 
     
   It is a natural question to ask whether always 
     $\hyperbolicdimension(f)>\eventualhyperbolicdimension(f)$ when $f$ is of 
     disjoint type with $\eventualhyperbolicdimension(f)<2$. 
     This was recently answered in the negative by 
     Mayer and Zdunik~\cite{MayerZdunik2019}, who show that there is a disjoint-type
     entire function $f$ with 
     $\eventualhyperbolicdimension(f) = \hyperbolicdimension(f) = \theta(f) < 2$. Moreover,
     there is no conformal (with respect to the cylinder or spherical metric) probability
     measure of exponent $\theta(f)$ that is supported on the radial Julia set. 
\end{remark}

\subsection*{A lower  bound on the eventual hyperbolic dimension}

 We now give  the details concerning the facts, stated in the  introduction, that
   the eventual hyperbolic dimension of a function  $f\in\classb$ is at least $1$, and that
   this lower bound is achieved by some of the functions studied by Mayer and Urba\'nski.

\begin{lem}[Eventual hyperbolic dimension in $\classb$]
  \label{lem:dim of bounded orbits at least 1}Let $f\in\classb$, and let 
   $R > 0$. Then 
  \[ \sup_X  \hausdorffdimension  (X) > 1, \]
  where the supremum is taken over all hyperbolic sets $X\subset \C\setminus D(0,R)$. 
   In particular, 
  \[ \vanishingexponent(f) \geq \eventualhyperbolicdimension(f) \geq 1. \]
\end{lem}
\begin{proof}
  \lrg{We may assume that $R$ is sufficiently large to ensure that $S(f)\subset D(0,R/2)$. 
    Then every component of $f^{-1}(\C\setminus \overline{D(0,R)})$ is a logarithmic
    tract of $f$ over infinity, and the first part follows from~\cite{BaranskiKarpinskaZdunik2009}.
 The second part is a direct consequence
  of the first, together with
   the definition of $\eventualhyperbolicdimension(f)$, and Lemma~\ref{lem:above ev hyp dim}.}
\end{proof}

  Let us use the vanishing exponent to justify the claim in the 
    introduction that entire functions $f\in\classb$ covered by the results of 
    \cite{MayerUrbanski2008} have eventual hyperbolic dimension $1$. 
  It is worth noting that \cite{MayerUrbanski2008} treats more than just entire functions.
  Their results cover finite-order meromorphic functions. In our work we restrict to the entire case.
      Furthermore, \cite{MayerUrbanski2008} does not explicitly restrict to functions in $\classb$, 
      although we are not aware of any known examples of entire functions outside of $\classb$ to which their results apply. 
    An entire function $f\in\classb$ of finite positive order $\rho$
     has \emph{balanced growth} in the sense of Mayer-Urba\'nski  if
     there exists $C > 1$
     such that for all $z\in \juliaset(f)$,
    \begin{equation}
      C^{- 1} \leq \frac{| f' (z) |}{(1 + | z |^{\rho - 1}) (1 + | f (z) |)}
      \leq C \label{eq:balanced-growth}
    \end{equation}
(see \cite[Lemma 3.1]{MayerUrbanski2008}).

  \begin{prop}[Eventual hyperbolic dimension and balanced growth]
    \label{prop:in mu class evhypdim is 1}
    Let $f$ be a class $\classb$ entire
    function of finite order $ \rho$.
    Suppose that $f$ has balanced
    growth in the sense of Mayer-Urba{\'n}ski.
    Then
    $\eventualhyperbolicdimension(f)  = \theta(f) = 1 < \hypdim(f)$.
  \end{prop}
\begin{proof}
  By Lemma~\ref{lem:dim of bounded orbits at least 1}, it only remains to show that
   $\vanishingexponent(f)\leq 1$. \lrg{The} balanced growth condition~{\eqref{eq:balanced-growth}}
    means precisely that $\cylnorm{\Deriv f(z)}$ is comparable to 
      $\max(|z|,1)^{\rho}$ for $z\in \juliaset(f)$, where $ \rho $ is the order of growth of $ f $.
      So 
  \begin{equation}\label{eqn:comparison}
      \lrg{\cylpartfun(t,f,w)=}
     \familysum{z \in f^{- 1} (w)}{} \cylnorm{\Deriv f (z)}^{- t} \approx
     \familysum{z \in f^{- 1} (w)}{} \frac{1}{| z |^{\rho t}}, \end{equation}
    for sufficiently large $w\in J(f)$.
     (Here $\approx$ means that the two sides are uniformly comparable.)

    Fix $\tilde{t}>t>1$. By~{\cite[Proposition~4.5]{MayerUrbanski2008}}, the second sum 
        in~\eqref{eqn:comparison} remains bounded
      as  $w\to\infty$. Hence
      \lrg{$ \sum_{z \in f^{- 1} (w)} 1/| z |^{\rho \tilde{t}} $}
    tends to zero as $w\to\infty$. 
    We conclude that 
    \begin{equation} \label{eq: first partition function converges to 0 in the Julia set}
         \lrg{\cylpartfun(\tilde{t},f,w)}\to 0,
     \end{equation}
     at least as $ w \to \infty $ inside $ J (f) $. Since $f\in\classb$, 
      every component of  $J(f)$ is unbounded
      (\cite[Proposition 2]{EremenkoLyubich1992}, \cite[Corollary~3.11]{RempeGillen2018}). 
      A simple application of Koebe's distortion theorem~-- which we provide below in
      Lemma~\ref{lem:partitiondistortion} for completeness~-- shows that
      $\cylpartfun(\tilde{t},f,w))$ and $\cylpartfun(\tilde{t},f,\tilde{w}))$ are
      comparable when $\lvert w \rvert = \lvert \tilde{w}\rvert$ and
      this modulus is sufficiently large. Therefore~\eqref{eq: first partition function converges to 0 in the Julia set} also holds for the unrestricted limit $w\to\infty$, and 
     $ \vanishingexponent(f) \leq \tilde{t}$.
     As  $\tilde{t}>1$ was arbitrary,
      the proof is complete.
\end{proof}

\begin{lem}[Distortion bounds on partition functions]\label{lem:partitiondistortion}
  Let $f\in\classb$ and $C>1$. Then there is $R>0$ such that 
    \[ \frac{\cylpartfun(t,f,w)}{\cylpartfun(t,f,\tilde{w})} \leq C^t \]
   whenever $t\geq 0$ and whenever $\lvert w\rvert = \lvert \tilde{w}\rvert \geq R$. 
\end{lem}
\begin{proof}
  Let $R_0 > 0$ so large that $S(f)\subset D( 0, R_0 / 2 )$ and $\lvert f(0)\rvert <R_0$. 
    Let $w,\tilde{w}\in\C$ with $\lvert w\rvert  = \lvert \tilde{w} \rvert \geq R_0 + 4\pi$.    Let $\gamma$ be an arc of the circle $\partial D(0,\lvert w\rvert)$ 
    connecting $w$ to $\tilde{w}$. For each
    $z \in f^{-1}(w)$, let  $\tilde{z}$ be the element of  $f^{-1}(\tilde{w})$ obtained from  $z$ by
    analytic continutation of $f^{-1}$ along $\gamma$. Observe that this defines a 
    bijection between $f^{-1}(w)$ and $f^{-1}(\tilde{w})$. By the definition 
    of the partition function, it is enough to show that 
     \begin{equation}\label{eqn:distortionC} \frac{\cylnorm{\Deriv f (z)}}{\cylnorm{\Deriv f(\tilde{z})}} \leq C \end{equation}
     when $\lvert w\rvert $ is sufficiently large. 

  This follows by applying Koebe's theorem to a \emph{logarithmic transform} of $f$.
    More precisely, set $W\defeq\C\setminus \overline{D(0,R_0)}$, let
    $T$ be the connected component of  $f^{-1}(W)$ containing
     $z$ and $\tilde{z}$, and let $V$ be a connected component of
     $f^{-1}(T)$. Recall that $0\notin T$, so $V$ is a simply-connected domain,
      and $f\circ \exp \colon V\to W$ is a universal covering. So there is a conformal isomorphism
     $\Phi\colon V\to H\defeq \log W$ such that $\exp\circ \Phi = f\circ \exp$. 
     
     Let $\zeta,\tilde{\zeta}\in V$ with $\exp(\zeta)=z$, $\exp(\tilde{\zeta})=\tilde{z}$.
       Then 
          \[ \cylnorm{\Deriv f (z)} = |(\Phi^{-1})'(\zeta)|, \]
        and similarly for $\tilde{\zeta}$. Also observe that 
        $\re \Phi(\zeta) = \re \Phi(\tilde{\zeta}) = \log |w|$, and 
        $|\im \Phi(\zeta) - \im \phi(\tilde{\zeta})| < 2\pi \leq (\lvert w\rvert-R_0)/2 $. 
        By Koebe's distortion theorem, applied to the restriction of
          $\Phi^{-1}$ to the disc $D(\Phi(\zeta), \lvert w\rvert-R_0)$, 
          we see that~\eqref{eqn:distortionC} holds
          with $C$ replaced by 
          a constant $C(\lvert w\rvert)$ that tends to $1$ as $\lvert w\rvert \to\infty$. 
          This completes the proof. 
\end{proof}

\subsection*{Affine and quasiconformal equivalence classes}

Following \cite{RempeStallard2010}, we show that, in $\classb$, the eventual hyperbolic dimension is
 invariant under  affine equivalence.

\begin{thm}\label{thm:affine equiv same evhypdim} 
Let $f,g\in\classb$ be affinely equivalent. Then
    $\evhypdim(f)=\evhypdim(g)$ and $\vanishingexponent(f)=\vanishingexponent(g)$. 
\end{thm}
\begin{proof}
  The second claim  is immediate from the definition. We show that 
    $\eventualhyperbolicdimension(g)\geq\eventualhyperbolicdimension(f)$; the first claim then follows trivially.
  
  Let $K>1$. By \cite[Corollary~2.2]{RempeStallard2010} 
  (see also \cite[Section~3]{Rempe2009}), there exist
  $R>0$ and a $K$-quasiconformal map $\theta:\complexnumbers\rightarrow\complexnumbers$
  such that
  \begin{equation}\label{eq:qc conj for affine equivalence}
    \theta(f(z)) = g(\theta(z))
  \end{equation}
  for all \lrg{$z\in \juliaset_R(f)=\left\{z\in\juliaset\colon |f^{\iterated n}(z)|\geq R \text{ for all $n\geq0$}\right\}$}.

Let $S>0$ and suppose $R>0$ is large enough, so that $|\theta(z)|\geq S$ for all $|z|\geq R$.
From the conjugacy \eqref{eq:qc conj for affine equivalence}, it follows that
  if $X\subset\juliaset_R(f)$ is an hyperbolic subset for $f$,
  then $Y\defeq\theta(X)$ is a compact subset of $\juliaset_S(g)$
  which is forward invariant by $g$.
  \lrg{By~\eqref{eqn:eremenkolyubichestimate}, $Y$ is a hyperbolic set for $g$, assuming $S$ was chosen
   sufficiently large.}

 From Theorem 8 of \cite{GehringVaisala1973}, one deduces the following inequality:
   \adz{\begin{equation*}
    \sup \left\{ \hausdorffdimension (X)\colon X\subset\juliaset_R(f)\mbox{ hyperbolic}\right\}
    \leq K \sup\left\{ \hausdorffdimension (Y)\colon Y\subset\juliaset_S(g)\mbox{ hyperbolic}\right\}.
  \end{equation*}}
  This implies that \adz{$\eventualhyperbolicdimension(f)\leq K \eventualhyperbolicdimension(g)$}.
  Since $K$ can be chosen arbitrarily close to $1$, the result follows.
\end{proof}

 As stated in Corollary~\ref{cor:noninvariance of evhypdim inside qc classes}, \lrg{the
   eventual  hyperbolic dimension} is \emph{not} invariant under quasiconformal equivalence,  
   and the proof will show that the same is true for the  vanishing exponent. 
    However, the condition of having \emph{full} eventual hyperbolic dimension is preserved under
    quasiconformal equivalence.  

\begin{prop}[Quasiconformal equivalence and eventual hyperbolic dimension]
  Let $f,g\in \classb$ be quasiconformally equivalent. If $\eventualhyperbolicdimension(f)=2$, then
    $\eventualhyperbolicdimension(g)=2$. 
\end{prop}
\begin{proof}
  This follows again from~\cite{Rempe2009}. \lrg{We begin as in the proof of Theorem~\ref{thm:affine equiv same evhypdim}. 
    The quasiconformal map  $\theta$ still exists in this setting \cite[Theorem~1.1]{Rempe2009}, 
    but it is no longer necessarily true that $K$ can be chosen arbitrarily close to $1$.
     Let $S$ and $R$ be as in the above proof, and let 
     $\mathcal{X}$ be the union of all hyperbolic sets of $f$ contained in  $J_R(f)$. 
     Then $\mathcal{Y}\defeq \theta(\mathcal{X})\subset J_S(g)$ is a union of hyperbolic sets of $g$.}

   \lrg{Since $\eventualhyperbolicdimension(f)=2$, we have
     $\dim(\mathcal{X})=2$ by definition. As quasiconformal maps preserve sets  of Hausdorff
     dimension $2$ \cite[Corollary~13]{GehringVaisala1973}, it follows that
      $\dim(\mathcal{Y})=2$. Since $S$ can be chosen
      arbitrarily large, we conclude that $\eventualhyperbolicdimension(g)=2$, as claimed.}
\end{proof} 

\begin{remark} \label{rem:is maximal vanishing exponent preserved under qc equiv}
  It is natural to ask whether the condition $\theta(f)=2$ is also preserved under quasiconformal equivalence.
 It seems plausible that this is the case, at least for functions with a finite set of singular values.
 Observe that this question is similar in spirit to~\cite[Proposition 4.2]{EpsteinRempeGillen2015}; we leave it aside
   as it will not be required for the purposes of this paper. 
\end{remark}

\section{Poincar\'e functions}

 In this section, we review some basic facts about the mapping properties of Poincar\'e functions.
Recall that a Poincar{\'e} function, or lineariser, is an entire solution of the equation
$\poincarefunction ( \multiplier z ) = f ( \poincarefunction(z) )$,
 with $\poincarefunction (0)=\xi_0 $.
Here $f$ is an entire function with a repelling fixed point $\fixedpoint$ with multiplier $\multiplier$.
 If one fixes the value of $\poincarefunction'(0)$ to some non zero complex number,
 then the solution is unique.
 The unique Poincar{\'e} function such that $\poincarefunction'(0)=1$ is called
 the \emph{normalised Poincar{\'e} function} of $f$
  at $\fixedpoint$ .

\begin{rem}[Normalisation] \label{rem:compute evhypdim only for normalised poincare functions}
\adz{Let $ f, \xi_0, L $ be as above and let $ \ph $ be a nonconstant affine map and $ \lambda \in \nonzerocomplexnumbers $.
 Then the mapping $ f_1 \defeq \ph \composedwith f \composedwith \ph^{ - 1 } $
 has a repelling fixed point at $ \xi_1 \defeq \ph (\xi_0) $
 with multiplier $ \rho \defeq f' (\xi_0) $.
 Then the function $ L_1 (z) \defeq \ph \composedwith L (\lambda z) $
 is a Poincar{\'e} function for $ (f_1, \xi_1) $.
    In particular, if the postsingular  set of $f$ (and hence of $f_1$) is bounded, then 
    it follows from Theorem \ref{thm:affine equiv same evhypdim} that
     $\eventualhyperbolicdimension(\poincarefunction_1)= \eventualhyperbolicdimension(\poincarefunction)$ and
     $\theta(\poincarefunction_1) = \theta(\poincarefunction)$. 

     For example if  $f$ and  $\fixedpoint$ are as above and $\poincarefunction$ is a Poincar\'e function of
     $f$ at $\fixedpoint$, then \adz{
     \[ z\mapsto L \p{\frac{z}{L'(0)}} - \fixedpoint \]}
     is  the normalised Poincar\'e function of $z\mapsto f(z+\fixedpoint)-\fixedpoint$ at $0$. 
    Therefore, in the  following, we can usually assume that $\fixedpoint=0$ and that a given 
    Poincar\'e function is normalised. }
\end{rem}

\begin{rem}[Quasiconformal equivalence of linearisers]\label{rem:qc conj implies qc equiv of lin}
  Similarly, 
   if two entire functions are quasiconformally conjugate, then their respective Poincar{\'e} functions
  are quasiconformally equivalent\adz{, see \cite[Proposition 3.2]{EpsteinRempeGillen2015}}.
   This will be important in Corollary~\ref{cor:noninvariance of evhypdim inside qc classes}.
\end{rem}

 Let $\poincarefunction$ be a Poincar\'e function of an entire function $f$.
  Then we can use the defining functional relation to 
  describe the preimage $L^{-1}(w)$ of a point $w\in\C$ in terms of the
  iterated preimages of $w$ under $f$, as follows. 
  (See Figure~\ref{fig:preimage structure}.)

\begin{obs}[Preimage structure of linearisers] \label{obs:preimage structure}
Let $f$ be an entire function with a repelling fixed point $\fixedpoint$ of multiplier $\multiplier$, and 
  let $\poincarefunction$ be a Poincar{\'e} function of $ f $ at $\fixedpoint$.
  Let $ \Delta_0 = \disk (0, \abs{ \rho } r_0) $ be a small disc around $0$ on which $\poincarefunction$ is univalent,
   and set \begin{align}
    A_0 & \defeq \setof{ z \in \complexnumbers \colon r_0 \leq |z| < r_0 |\multiplier| } \subset \Delta_0 \\
     A_f & \defeq \poincarefunction(A_0).
           \end{align}
  In particular $A_0$ is a fundamental domain for $ z\mapsto \multiplier z $.
  
Let $w\notin \poincarefunction(\Delta_0)$ and $n\geq 1$ and
define
\begin{equation} \label{lem:preimage structure,item:level set}
 \poincarelevset_n \defeq \poincarefunction^{-1} (w)
    \intersection
    \setof{ z \in \complexnumbers \colon \abs{ \multiplier }^{ n } \cdot r_0 \leq |z| < |\multiplier|^{ n + 1 } \cdot r_0 }.
\end{equation}
Observe that
$ \poincarefunction^{- 1} (w) = \familyunion{n\geq 1}{} \poincarelevset_n $
and
\begin{align}
 \poincarelevset_n =
 \multiplier^n\cdot  (\poincarefunction_{\restricted \Delta_0})^{-1} \left( f^{- n} (w) \setintersection A_f \right).
\end{align}

Finally, let $z \in \poincarelevset_n$ and 
      set $\zeta \defeq \poincarefunction(z/\multiplier^n) \in f^{- n} (w)\setintersection A_f$. Then
\begin{equation} \label{lem:preimage structure,item:derivative}
 \poincarefunction' (z) = \frac{1}{\multiplier^n} L' \left( \frac{z}{\multiplier^n} \right) \left({f^{\iterated n}}\right)' (\zeta).
\end{equation}
\end{obs}

\begin{figure}
 \setlength{\fboxsep}{0pt}%
 \setlength{\fboxrule}{1pt}%
 \fbox{\includegraphics[scale=0.17]{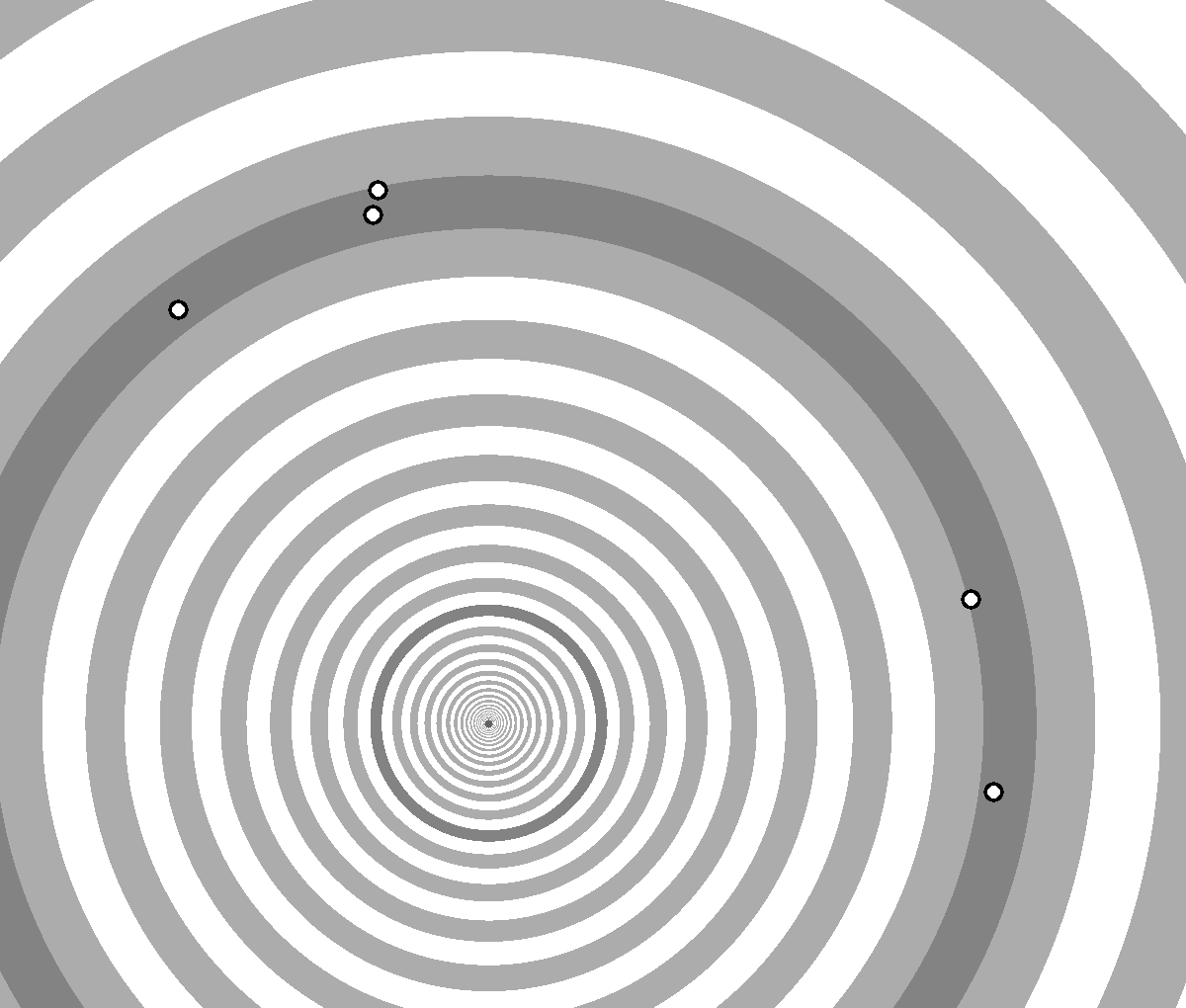}}
 \fbox{\includegraphics[scale=0.17]{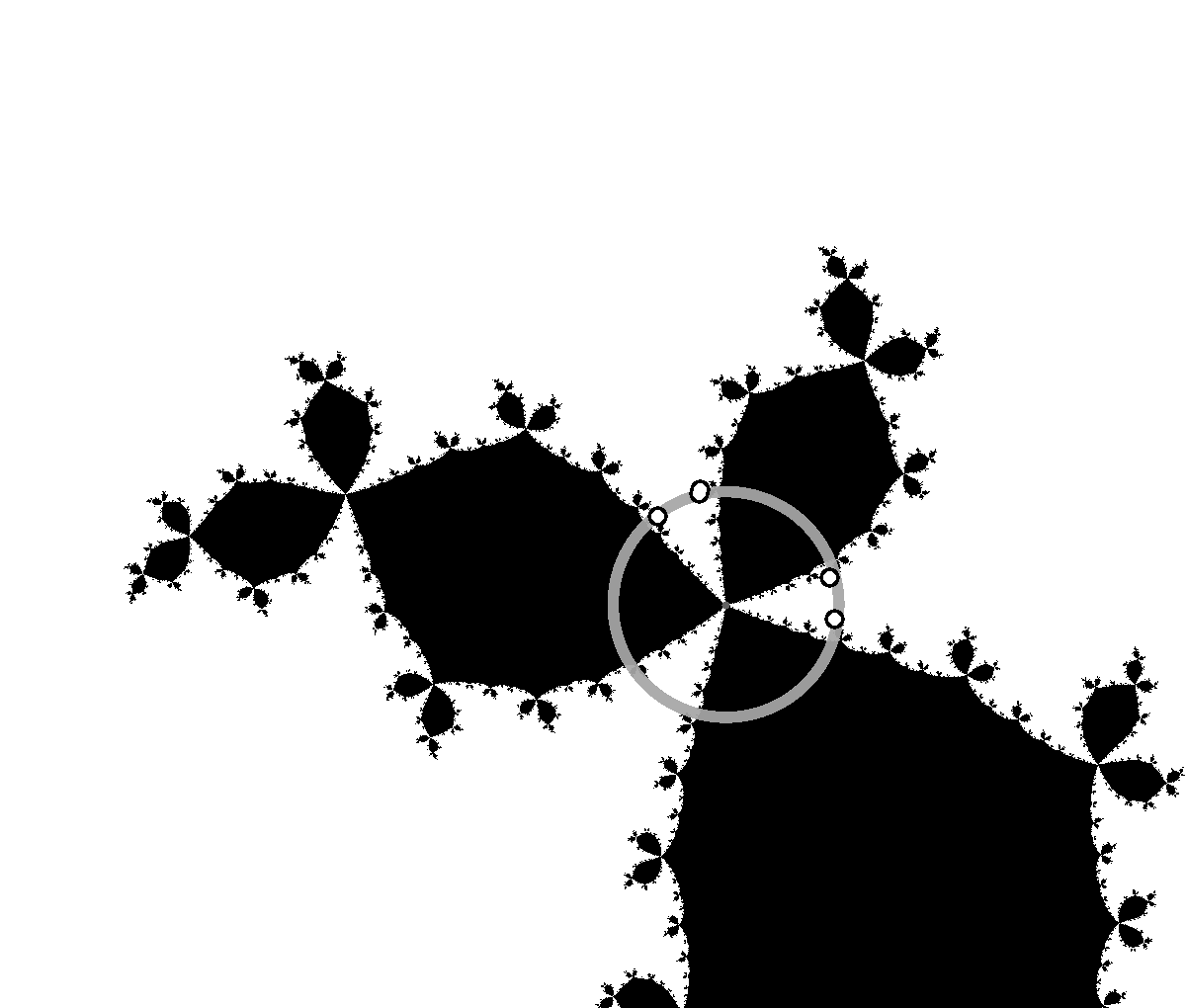}}
 \caption{\label{fig:preimage structure}%%
    Illustration of Observation \ref{obs:preimage structure}, for a lineariser $\poincarefunction$
          at a fixed point $\alpha$ of the Douady rabbit polynomial $P(z)=z^2+c, c\approx 0.123+0.745\sqrtmone$.
          The domain of $L$ is  on the left, while the range (the dynamical plane of $P$) is on the right. 
          An approximation of a  fundamental annulus for the dynamics of $P$ is shown in grey on the right. 
          The five preimages of $w\defeq 2+2\sqrtmone$ by $P^{\iterated 15}$
          which belong to the annulus are represented by white discs with black borders.
          (Two of these preimages lie very close together.) The filled Julia set of $P$ is shown in black in the background.  
          On the left, $A_0$ is drawn as a small dark grey round annulus with its centre at $0$.
          The different annuli $|\multiplier|^{k}A_0$ are differentiated by white and light grey colors.
          The annulus $|\multiplier|^{15}A_0$ is also \adz{represented} in dark grey. 
          By Observation~\ref{obs:preimage structure}, 
          this annulus contains exactly five elements of $\poincarefunction^{-1}(2+2\sqrtmone)$,
          corresponding to the five points shown on the right;  
          these are the elements of  $\poincarelevset_{15}$.} % where $w=2+2\sqrtmone$.
\end{figure}

\begin{proof}
 The claims follow immediately from the functional relation~\eqref{eq:schroederequation}. 
\end{proof}

We will need the following result (compare also \cite[p.~581,~Footnote~2]{EpsteinRempeGillen2015}).

\begin{prop}[{\cite[Proposition 4.2 (ii)]{MihaljevicBrandtPeter2012}}]\label{prop:sing set of linearisers}
Let $f$ be an entire function and let $\poincarefunction$ be a Poincar\'e function of $f$. Then $S(\poincarefunction)=\postsingular(f)$. 
\end{prop}

\section{A lower bound}

The following result, together with Theorem \ref{thm:affine equiv same evhypdim},
 implies Theorem \ref{thm:Eventual hyperbolic dimension of linearisers II}.

\begin{thm}
  \label{thm:evhypdim ge hypdim}
  Let $f$ be a non-constant entire function which is not affine. Let $\poincarefunction$ be the normalised Poincar{\'e} function
  associated to a repelling fixed point of $f$. Then
  \[ \eventualhyperbolicdimension  \poincarefunction \geq \hyperbolicdimension f. \]
\end{thm}

In the proof of this theorem,
 we construct a sequence of closed round annuli $\lbAnnulus_\lbAnnulusidx$,
 centered at $0$ and with distance to $0$ converging to $\infty$,
 and satisfying the following: for any $t<\hyperbolicdimension f$,
 the first partition function of
 $\poincarefunction_{\restricted \lbAnnulus_\lbAnnulusidx\setintersection\poincarefunction^{-1}(\lbAnnulus_\lbAnnulusidx)}$ 
 with exponent $t$
 will be bounded from below by a large constant as long as $\lbAnnulusidx$ is large enough.

 From the Eremenko-Lyubich estimate \eqref{eqn:eremenkolyubichestimate}
 it follows that
 the set $X$ of points whose orbits stay inside the annulus $\lbAnnulus_\lbAnnulusidx$
 is a nonempty hyperbolic set for $\poincarefunction$ for $\lbAnnulusidx$ large enough.
 Since the first partition function for this hyperbolic set is bounded from below
 by a constant which can be chosen as large as necessary (again with $\lbAnnulusidx$ large enough),
 the dimension of the set $X$ is at least $t$.
 This implies $ \eventualhyperbolicdimension L \geq t $ for any $t<\hyperbolicdimension f$,
 and hence the stated conclusion.

We start with the basic construction of the sequence $(\lbAnnulus_\lbAnnulusidx)_\lbAnnulusidx$.
 This is done in the following lemma.

\begin{lem} \label{lem:wiman-valiron-eremenko construction}
Let $f$ be a transcendental entire function or a polynomial of degree $d\geq 2$.
Let $\lbFundradius>0$, $\lbAnnmod>2$ and $\multiplier$ a complex number such that $|\multiplier|>1$.
Let $\lbLargemod=4+2\lbAnnmod\log|\multiplier|$
 Then there exists a sequence of natural numbers $\lbAnnlevel_\lbAnnidx\tendsto\infty$,
 and a sequence of bounded open subsets of the complex plane $\lbSector_\lbAnnulusidx$ \adz{(with an explicit description given below)}
 satisfying the following.
 Let 
 \begin{equation}\label{eq:lb, def of annuli}
  \lbAnnulus_\lbAnnidx =
       \annulus \left( \left| \multiplier \right|^{\lbAnnlevel_\lbAnnidx} \lbFundradius,
                       \left| \multiplier \right|^{\lbAnnlevel_\lbAnnidx + \lbAnnmod} \lbFundradius \right)
 \end{equation}
 then,
 \begin{itemize}
  \item there is a sequence of complex numbers $\lbWVpointseq_\lbAnnulusidx\tendsto\infty$
  and a nondecreasing sequence of natural numbers $\lbWVcentralidx_\lbAnnulusidx$ such that
  \begin{equation} \label{eq:lb, def of sector}
  \lbSector_\lbAnnulusidx=\setof{ z :
  \left| \log \left| \frac{z}{\lbWVpointseq_\lbAnnulusidx} \right| \right| < \frac{\lbLargemod}{\lbWVcentralidx_\lbAnnulusidx},
  \left| \arg \frac{z}{\lbWVpointseq_\lbAnnulusidx} \right| < \frac{\lbLargemod}{\lbWVcentralidx_\lbAnnulusidx} };
 \end{equation}
  \item For all $\lbAnnidx\geq 0$, $f(\lbSector_\lbAnnulusidx)\supset\lbAnnulus_\lbAnnidx\setunion \lbSector_{\lbAnnulusidx+1}$\adz{ ;
  \item $ \lbWVpointseq_{ \lbAnnulusidx + 1 } \in \lbAnnulus_\lbAnnidx $.}
 \end{itemize}
 Moreover Wiman-Valiron theory applies at $\lbWVpointseq_\lbAnnulusidx$
 on
 $\lbSector_\lbAnnulusidx$
 with central index $\lbWVcentralidx_\lbAnnulusidx$,
 that is,
 \begin{enumerate}
  \item \adz{ $|f(\lbWVpointseq_\lbAnnulusidx)|=\maximummodulus(r_\lbAnnulusidx, f) \defeq \max \setof{ \abs{ f(z)} : \abs{ z } = r_k } $,}
  with $r_\lbAnnulusidx=|\lbWVpointseq_\lbAnnulusidx|$;
  \item For all
  $z\in \lbSector_\lbAnnulusidx$
    \begin{equation} \label{eq: wiman-valiron value of the function}
      f (z)
       =
      \left( \frac{z}{\lbWVpointseq_\lbAnnulusidx} \right)^{\lbWVcentralidx_\lbAnnulusidx}
            f (\lbWVpointseq_\lbAnnulusidx )
            \left( 1 + \lbWVerror_0 (z) \right),
    \end{equation}
    and
    \begin{equation}
      f' (z) = \frac{\lbWVcentralidx_\lbAnnulusidx}{z}
               \left( \frac{z}{\lbWVpointseq_\lbAnnulusidx} \right)^{\lbWVcentralidx_\lbAnnulusidx}
                f (\lbWVpointseq_\lbAnnulusidx) \left( 1 + \lbWVerror_1 (z) \right), \label{eq:wiman-valiron derivative}
    \end{equation}
    with $\left| \lbWVerror_0 (z) \right| + \left| \lbWVerror_1 (z) \right| \leq \lbWVerror_\lbAnnulusidx$
    for all $z\in \lbSector_\lbAnnulusidx$
    and $\lbWVerror_\lbAnnulusidx\tendsto 0$;
    \item We have the following estimate:
    \begin{equation}
      \lbWVcentralidx_\lbAnnulusidx \leq \left( \log \maximummodulus (r_\lbAnnulusidx, f) \right)^2.
      \label{eq:wiman-valiron central index estimate} 
    \end{equation}
 \end{enumerate}
\adz{
 Finally,
 \begin{equation} \label{eq: estimate on the size of the WV annulus}
  \frac{ 1 }{ \abs{ \rho }^2 R_f } \maximummodulus (r_k, f)
  \leq
  \abs{ \rho }^{ n_k }
  \leq
  \frac{ 1 }{ \abs{ \rho } R_f } \maximummodulus (r_k, f).
 \end{equation}
}
\end{lem}

\begin{proof}
 We do a similar construction as the one that can be found in the early reference \cite{Eremenko1989}.
 
For any $r \geq 0$, denote by $\lbWVpoint (r)$ a definite arbitrary choice of a complex number such that $| \lbWVpoint (r) | = r$
 with $| f (\lbWVpoint (r)) | = \maximummodulus (r, f)$.
Let $\lbLargemod>1$.
From Wiman-Valiron theory, we know that there exists an exceptional set $\lbWVexceptions \subset \realnumbers_+$
  of finite logarithmic measure (that is $\logarithmicmeasure \lbWVexceptions \defeq \int_\lbWVexceptions \frac{d x}{x}<\infty$),
  a nondecreasing function
  $r\mapsto \lbWVcentralidx (r)$
  from \adz{$ \positiverealnumbers\setcomplement\lbWVexceptions$ } to $\naturalnumbers$,
  and a positive function $r \mapsto \lbWVerror (r)$
  converging to $0$ as $r \tendsto \infty$, such that the following is
  satisfied for all
  \adz{$r \in \positiverealnumbers \setcomplement \lbWVexceptions$:}
  \begin{enumerate}
    \item For all \adz{ $z \in \disk \left( \lbWVpoint (r), \frac{ 2 \lbLargemod r}{\lbWVcentralidx (r)} \right)$,}
    \[ f (z) = \left( \frac{z}{\lbWVpoint (r)} \right)^{\lbWVcentralidx (r)} f (\lbWVpoint (r)) \left( 1 + \lbWVerror_0 (z) \right), \]
    and
    \begin{equation*}
      f' (z) = \frac{\lbWVcentralidx (r)}{z}
      \left( \frac{z}{\lbWVpoint (r)} \right)^{\lbWVcentralidx (r)}
      f (\lbWVpoint (r)) \left( 1 + \lbWVerror_1 (z) \right),
    \end{equation*}
    with $\left| \lbWVerror_0 (z) \right| + \left| \lbWVerror_1 (z) \right| \leq \lbWVerror (r)$
    for all $z$ in that disk.
        
    \item We have the following estimate:
    \begin{equation*}
      \lbWVcentralidx (r) \leq \left( \log \maximummodulus (r, f) \right)^2 .
    \end{equation*}
  \end{enumerate}
 In the case where $f$ is a polynomial map of degree $d\geq 2$,
 the above is still true if
 the central index $\lbWVcentralidx$ is replaced by the constant equal to $d$
 and the exceptional set $\lbWVexceptions$ is some bounded subset of \adz{$\positiverealnumbers$} depending on $f$.

\adz{Now recall that} $ M = 4 + 2 \lbAnnmod \log|\multiplier|$.
With this, one can construct a sequence of points $\lbWVpointseq_\lbAnnulusidx$
 and of sectors
 \adz{ $\lbSector_\lbAnnulusidx
   \subset
   \disk \left( \lbWVpointseq_\lbAnnulusidx, \frac{ 2 \lbLargemod r_\lbAnnulusidx}{\lbWVcentralidx_\lbAnnulusidx} \right)$ }
 defined by equation \eqref{eq:lb, def of sector}
 where
 $\lbWVpointseq_\lbAnnulusidx = \lbWVpoint (r_\lbAnnulusidx)$,
 $r_\lbAnnulusidx=| \lbWVpointseq_\lbAnnulusidx | \notin \lbWVexceptions$,
 $\lbWVcentralidx_\lbAnnulusidx = \lbWVcentralidx (r_\lbAnnulusidx)$
 and which is such that
 $f (\lbSector_\lbAnnulusidx)$ contains both $\lbSector_{\lbAnnulusidx + 1}$
 and the round annulus
 \adz{$\lbAnnulus_\lbAnnulusidx
  \ni\lbWVpointseq_{\lbAnnulusidx + 1}$},
  where $\lbAnnlevel_\lbAnnidx \in \naturalnumbers$.
  \adz{For $ f $ transcendental the ratio $ M / N(r) $ tends to $ 0 $ as $ r \tendsto \infty $
  and thus we can assume that $ e^{ M / N (r) } \leq 1 + \frac{ 2 M }{ N (r) } $
  for all $ r $ large enough. In the polynomial case $ M / N(r) $ is still bounded and it is possible to find a $ \kappa > 1 $
  large enough so that $ e^{ M / N (r) } \leq 1 + \frac{ \kappa M }{ N (r) } $ and work with $ \frac{ \kappa M }{ N (r) } $ instead of $ \frac{ 2 M }{ N (r) } $. }

Choose first $\lbRadseq_0 \notin \lbWVexceptions$
 and let $\lbWVpointseq_0 = \lbWVpoint (\lbRadseq_0)$ and
 \begin{equation*}
 S_0 = \left\{ z :
  \left| \log \abs{ \frac{z}{\lbWVpointseq_0} } \right| < \frac{\lbLargemod}{\lbWVcentralidx_0},
  \left| \arg  \frac{z}{\lbWVpointseq_0}  \right| < \frac{\lbLargemod}{\lbWVcentralidx_0} \right\}.
 \end{equation*}

Suppose now $\lbWVpointseq_\lbAnnidx$,
 and thus $\lbSector_\lbAnnidx$ have already been defined.
The mapping
 $$\ph_\lbAnnidx : z \mapsto \lbWVcentralidx_\lbAnnidx \,(\log z  - \log \lbWVpointseq_\lbAnnidx) + \log f (\lbWVpointseq_\lbAnnidx)$$
 sends $\lbSector_\lbAnnidx$ univalently onto the square
  $$\lbSquare_\lbAnnidx = \left\{ u :
      - \lbLargemod < \realpart (u - \log f (\lbWVpointseq_\lbAnnidx)) < \lbLargemod,
      - \lbLargemod < \imaginarypart (u - \log f (\lbWVpointseq_\lbAnnidx)) < \lbLargemod
     \right\}.$$
Assuming $\lbRadseq_0$ large enough,
  $\left| \log f - \ph_\lbAnnidx \right| = \left| \log \left( 1 + \lbWVerror_0 \right) \right| < 1$
  is true for any $r\in\closedopeninterval{\lbRadseq_0,+\infty}\setcomplement\lbWVexceptions$.
It then follows from Rouch{\'e}'s theorem that there exists an open subset
  $U_\lbAnnidx \subset \lbSector_\lbAnnidx$
  which is sent by $\log f$ univalently onto the square
  $$\lbSquareimage_\lbAnnidx = \left\{ u :
     - \lbLargemod + 1 < \realpart (u - \log f (\lbWVpointseq_\lbAnnidx)) < \lbLargemod - 1,
     - \lbLargemod + 1 < \imaginarypart (u - \log f (\lbWVpointseq_\lbAnnidx)) < \lbLargemod - 1
   \right\}.$$
   In particular, the image of $\lbSector_\lbAnnidx$ by $f$
   contains the annulus
   $$\lbLargerann_\lbAnnidx = \left\{ z :
     e^{- \lbLargemod + 1} \leq \left| \frac{z}{f (\lbWVpointseq_\lbAnnidx)} \right| \leq e^{\lbLargemod - 1}
    \right\}.$$
  
Then, by definition of $\lbLargemod$, the integer
  $\lbAnnlevel_\lbAnnidx =
   \integerpart{\frac{\log | f (\lbWVpointseq_\lbAnnidx) / \lbFundradius |}{\log \left| \multiplier \right|}} - 1
   \in \naturalnumbers$
  is such that
  $\lbAnnulus_\lbAnnidx \subset \lbLargerann_\lbAnnidx$.
  \adz{Inequalities \eqref{eq: estimate on the size of the WV annulus}
  follow from this choice.}
  Note that for $\lbRadseq_0$ large enough, we will have $\lbAnnlevel_\lbAnnidx\tendsto\infty$.
Choosing \adz{$\lbRadseq_0 $ large enough, one has
  \adz{$\logarithmicmeasure \p{ \lbWVexceptions \intersection \openclosedinterval{ r_0, +\infty } } <
     \logarithmicmeasure  \left( \lbAnnulus_\lbAnnidx \setintersection \positiverealnumbers \right)$.}}
Hence there is a
  $\lbRadseq_{\lbAnnidx+1} \in
      \left( \lbAnnulus_\lbAnnidx \setintersection \positiverealnumbers \right) \setcomplement \lbWVexceptions$.
Also, one can suppose $\lbRadseq_0$ large enough so that
    $\lbWVcentralidx (r) \geq 2$
    for all $r \geq \lbRadseq_0$.
  As a consequence, one can find a $\lbWVpointseq_{\lbAnnidx + 1} \in  \lbAnnulus_\lbAnnidx$,
  with $| \lbWVpointseq_{\lbAnnidx + 1} | = \lbRadseq_{\lbAnnidx + 1}$ and
  $\lbSector_{\lbAnnidx + 1} \subset  \lbLargerann_\lbAnnidx$.
\end{proof}

Now we can use Lemma \ref{lem:wiman-valiron-eremenko construction} to prove Theorem \ref{thm:evhypdim ge hypdim}.

\begin{proof}[Proof of Theorem \ref{thm:evhypdim ge hypdim}] (See Figure \ref{fig:proof of lower bound}).
\adz{Let $\multiplier$ be the multiplier of the repelling fixed point.}
Let $A_0=\annulus(\lbFundradius,|\rho|\lbFundradius)$ 
 be the fundamental annulus from Observation \ref{obs:preimage structure}
\adz{ (i.e. $ R_f = r_0 / \abs{ \rho } $)}.
 Take $\ifsdomain\subset A_0$, $a,b,C,\beta>0$ and the sequence
 $\left(\lbIfsit_p,(\ifsdomain_i^{p})_{i=1,\dots,I_p}\right)_p$
 as in Lemma \ref{lem:lower exp growth}
 and let \adz{$ \lbAnnmod = a + 2 b + 2 $}.
 Lemma \ref{lem:wiman-valiron-eremenko construction}
 gives the sequences of annuli $\lbAnnulus_\lbAnnidx$,
 of natural numbers $\lbAnnlevel_\lbAnnidx$ and $\lbWVcentralidx_\lbAnnidx$,
 of sets $\lbSector_\lbAnnidx$,
 and of points $\lbWVpointseq_\lbAnnidx$.

The domain $\ifsdomain$ intersects the Julia set of $f$,
  hence the family $(f^{\iterated k})_{k\in\naturalnumbers}$
  is not normal on $ \ifsdomain $ \adz{ and for any bounded $ G $ there is $ k $ such that $ G \subset f^{\iterated k } \p{ \ifsdomain } $}.
  It follows from \adz{the Ahlfors island theorem that there exists
  a subdomain $\ifsdomain_0\subset\ifsdomain$,
  $ j \in \setof{ 0, 1, 2 } $ and $\lbCoverit\geq 1$
  such that}
  $f^{\iterated \lbCoverit}$ send $\ifsdomain_0$ univalently onto $\lbSector_j$.
If necessary, we remove the first terms of the sequence of $\lbSector_n$ so that $j=0$.

Hence, for all $w \in \lbAnnulus_{\lbAnnidx}$,
  there exists $\lbPreimginifs \in \ifsdomain_0$, such that,
  $f^{\iterated \lbCoverit + j} (\lbPreimginifs) \in \lbSector_j$
  for all $j = 0, \ldots ., \lbAnnidx - 1$
  and $f^{\iterated \lbCoverit + \lbAnnidx} (\lbPreimginifs) = w$.
  
For all $\lbAnnidx$, there is a integer $\lbIfsitact_\lbAnnidx\geq 0$, such that
\[
 \lbAnnlevel_\lbAnnidx + 1 - \lbAnnidx - \lbCoverit + b
    \leq a \lbIfsitact_\lbAnnidx \leq
 \lbAnnlevel_\lbAnnidx-\lbAnnidx-\lbCoverit + \lbAnnmod - 1 - b.
\]
From above and Lemma \ref{lem:lower exp growth}, it follows that
 one can choose a sequence of $p_k$ \adz{having the following properties.
 Let % we define
 \begin{equation*}
\lbIfsittransit_\lbAnnidx \defeq \lbIfsit_{p_\lbAnnidx} +\lbAnnidx+\lbCoverit,
\end{equation*}
 then
\[
 \lbAnnlevel_\lbAnnidx + 1\leq
 \lbIfsittransit_\lbAnnidx 
 \leq \lbAnnlevel_\lbAnnidx + \lbAnnmod - 1,
\]
inequality \eqref{pressure-explosion-on-the-ifs} is satisfied and for all $ { i = 1, \dots, I_{p_\lbAnnidx} } $,
$ f^{\iterated \lbIfsit _ {p _ \lbAnnidx}} _ {\restricted \ifsdomain ^ {p _ \lbAnnidx} _ i}: \ifsdomain_i \rightarrow \ifsdomain$
is a conformal isomorphism.}

Given $\lbAnnidx$ large choose an open set $\lbSlit_\lbAnnidx$ such that
 $\lbSlitAnnulus_\lbAnnidx\defeq\lbAnnulus_\lbAnnidx\setcomplement\lbSlit_\lbAnnidx$
 is simply connected and
 $\multiplier^{\lbIfsittransit_\lbAnnidx} \poincarefunction^{-1} \p{ \ifsdomain } \subset \lbSlitAnnulus_\lbAnnidx$.
 Choose $Y_k\subset\lbSector_0$
 a univalent preimage of
 $\lbSlitAnnulus_\lbAnnidx \setintersection \multiplier^{\lbIfsittransit_\lbAnnidx} A_0$
 by $f^{\iterated k}$.
 Consider the finite family of sets:
 \adz{
\begin{equation*}
 X_i^{\lbAnnidx} = 
  \multiplier^{ \lbIfsittransit_\lbAnnidx }
  \poincarefunction^{ - 1 } \p{
  \p{f ^ {\lbIfsit_{p_\lbAnnidx}}_{\restricted \ifsdomain_i^{p_\lbAnnidx}}}^{-1}
    \p{\p{f^{\lbCoverit}_{\restricted \ifsdomain}}^{-1} \p{Y_\lbAnnidx}} },
\end{equation*} }
 for $i=1,\dots,I _ {p_\lbAnnidx}$.
 
 Then $X_i^\lbAnnidx$ is a compact subset of
  $\lbSlitAnnulus_\lbAnnidx \setintersection \multiplier^{\lbIfsittransit_\lbAnnidx} A_0$
  and
  $\poincarefunction:
    X_i^\lbAnnidx
      \rightarrow
    \lbSlitAnnulus_\lbAnnidx \setintersection \multiplier^{\lbIfsittransit_\lbAnnidx} A_0$
  is a conformal isomorphism.
  
 We will show that the first partition function of the system of conformal isomorphisms
  $(\poincarefunction_{\restricted X_i^\lbAnnidx}, X_i^\lbAnnidx)_{i=1,\dots,I_\lbAnnidx}$
  taken with respect to the cylindrical metric
  becomes arbitrary large as $\lbAnnidx \tendsto \infty$.
This would imply that the (cylindrical) pressure function for the 
 IFS  $(\poincarefunction_{\restricted X_i^\lbAnnidx}, X_i^\lbAnnidx)_{i=1,\dots,I_\lbAnnidx}$
 is positive for $\lbAnnidx$ large.

Let $w\in \multiplier^{\lbIfsittransit_\lbAnnidx}A_0$
 and consider the first (modified) cylindrical partition function
\footnote{The modification on the metric can be ignored since,
  for $k$ large, $\abs{w}$ and the $\abs{z_i}$ are all greater than $1$.}
 of the system evaluated at $w$:
\begin{equation} \label{eq:LB 1st part fun}
 \partitionfunction(w) \defeq
 \familysum{i=1}{I_{p_\lbAnnidx}} \frac{ |w|^t }{ |z_i|^t \cdot \abs{ \poincarefunction'(z_i) }^t},
\end{equation}
where $z_i=\left(\poincarefunction_{\restricted X_i^\lbAnnidx}\right)^{-1}(w)$.

For all $z_i\in X_i^{\lbAnnidx}$, one has
\begin{equation} \label{eq:LB derivative of ifs}
\poincarefunction'(z_i)=
  \frac{1}{\multiplier^{\lbIfsittransit_\lbAnnidx}}
  \poincarefunction'\left(\frac{z_i}{\multiplier^{\lbIfsittransit_\lbAnnidx}}\right)
  {f^{\lbIfsit _ {p_\lbAnnidx}}}'\left(z_{IFS}^i\right)
  {f^{\lbCoverit}}'\left(z_{\lbSector_0}^i\right)
  {f^{\lbAnnidx}}'\left(z_{WV}^i\right),
\end{equation}
where
\begin{align*}
z_{IFS}^i         & =  \poincarefunction\left(\frac{z_i}{\multiplier^{\lbIfsittransit_\lbAnnidx}}\right),\\
z_{\lbSector_0}^i & =  f^{\iterated \lbIfsit_ {p_\lbAnnidx}}\left(z_{IFS}^i\right), \\
z_{WV}^i          & =  f^{\iterated \lbAnnidx}\left(z_{\lbSector_0}^i\right).
\end{align*}

By construction $|z_i|$ is comparable to $\left|\multiplier^{\lbIfsittransit_\lbAnnidx}\right|$.
 Since $\frac{z_i}{\multiplier^{\lbIfsittransit_\lbAnnidx}} \in A_0$,
 the factor
 $\left|\poincarefunction'\left(\frac{z_i}{\multiplier^{\lbIfsittransit_\lbAnnidx}}\right)\right|$
 is uniformly bounded away from $0$ and $\infty$.
 Finally, the same is true for the derivative of $f_{\restricted \ifsdomain}^{\lbCoverit}$.
 Hence $\partitionfunction(w)$ is comparable to
\begin{equation*}
 \familysum{ i = 1 }{ I_{ p_\lbAnnidx } }
   \frac
    {|w|^t}
    {
      \abs{ \p{ f^{\lbIfsit _{p_\lbAnnidx}} }' \p{z_{IFS}^i} \p{ f^{\lbAnnidx} }' \p{z_{WV}^i} }
    }.
\end{equation*}

From 
 Wiman-Valiron estimates \eqref{eq: wiman-valiron value of the function} and \eqref{eq:wiman-valiron derivative},
 if follows that for any $\lbPtinsec\in\lbSector_\lbAnnidx$, one has
 $ f' (\lbPtinsec) = (1 + \eps_2 (y)) \lbWVcentralidx_\lbAnnidx\frac{f(\lbPtinsec)}{\lbPtinsec}$
 for some bounded $ \eps_2 $.
 This implies,
\begin{equation*}
 \abs{ \p{ f^{\iterated \lbAnnidx} }' \p{ z_{WV}^i } }
  \leq
 C_1^{\lbAnnidx}
 \lbWVcentralidx_0 \lbWVcentralidx_1 \cdots \lbWVcentralidx_{\lbAnnidx - 1}
 |w|,
\end{equation*}
for some constant $C_1>0$.
Consequently, for $\lbAnnidx$ large enough,
\begin{align*}
\partitionfunction(w) &
  \geq 
 C_2^\lbAnnidx \,
 \frac{1}{\lbWVcentralidx_0^t \lbWVcentralidx_1^t \cdots \lbWVcentralidx_{\lbAnnidx - 1}^t} \,
 \familysum{ i = 1 }{ I_{ p_\lbAnnidx } }
  \left|{f^{\lbIfsit_ {p_\lbAnnidx}}}'\left(z_{IFS}^i\right)\right|^{-t} \\  
&
  \geq 
  C_3^\lbAnnidx \,
  \frac{1}{\lbWVcentralidx_0^t \lbWVcentralidx_1^t \cdots \lbWVcentralidx_{\lbAnnidx - 1}^t} \,
  e^{\beta \lbAnnlevel _\lbAnnidx},
\end{align*}
 for suitable constants $C_i>0$.

\adz{
From \eqref{eq: estimate on the size of the WV annulus}
it follows that there exists $ C_5 > 1 $ such that
 \begin{equation} \label{eq: WV maximum modulus comparable to annulus size}
  C_5^{ - 1 } \log \maximummodulus (r_k, f) \leq n_k \leq C_5 \log \maximummodulus (r_k, f).
 \end{equation}
In particular
from the above
 \eqref{eq:wiman-valiron central index estimate},
 and the fact that $ n_k \tendsto \infty $
 it follows that there exists $ C_6 > 0 $ such that
 $ N_j \leq C_6 n_{ j + 1 }^2 $ for all $ j \leq k $.
  Hence for $ j < k $,
  \begin{equation}
   \lbWVcentralidx_0 \lbWVcentralidx_1 \cdots \lbWVcentralidx_{\lbAnnidx - 1}
       \leq C_6^\lbAnnidx (\lbAnnlevel_1\cdots\lbAnnlevel_\lbAnnidx)^2.
  \end{equation}
  Since the map $ f $ is entire
  it follows from \eqref{eq: WV maximum modulus comparable to annulus size}
  that the ratio
  $ \frac{ e^{ \beta \lbAnnlevel_\lbAnnidx } }{ (\lbAnnlevel_1\cdots\lbAnnlevel_\lbAnnidx)^{ 2 t } } $
  grows to $ \infty $ as $ k \tendsto \infty $
  faster than any exponential growth in $ k $.
}
\end{proof}

\begin{figure}
 \def\svgwidth{\columnwidth}
  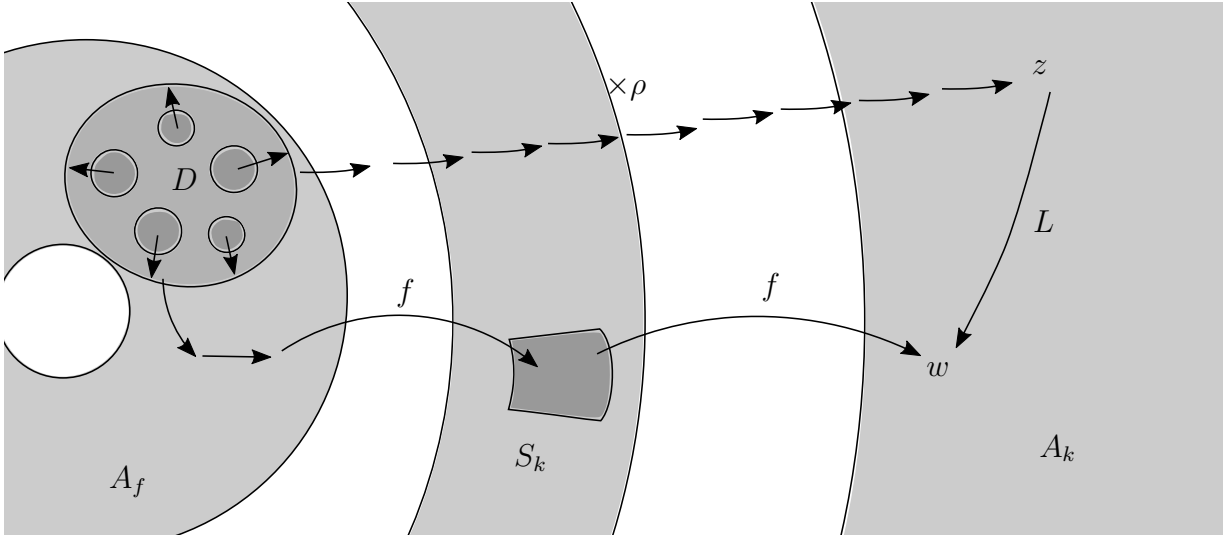
\caption{Illustration of the proof of theorem \ref{thm:evhypdim ge hypdim}.\label{fig:proof of lower bound}}
\end{figure}

\section{Linearisers of polynomials}

The following result, when combined with Theorem~\ref{thm:evhypdim ge hypdim},
 will establish Theorem~\ref{thm:evhypdim of linearisers}.

\adz{
\begin{lem} \label{lem: evhypdim L leq hypdim P for TCE J connected P}
 Let $ L $ be a Poincar{\'e} function for a topological Collet Eckmann polynomial $ P $ such that $ \juliaset (P) $ is connected.
 Then
 \begin{equation}
  \evhypdim(L) \leq \vanishingexponent(L) \leq \hypdim(P).
 \end{equation}
\end{lem}
}

\begin{proof}
   The first inequality is given by Lemma \ref{lem:above ev hyp dim}.
   Hence  we must only prove the second. 
   
  Fix $r_0$ small so that the restriction of $\poincarefunction$ to a
  neighborhood of the disk $D_0 = \Disk (0, r_0)$ is univalent. Let $A_0 = D_0
  \setcomplement \Disk (0, r_0 / | \multiplier |)$, $D_P = \poincarefunction
  (D_0)$, $A_P = D_P \setcomplement \PoincareFunctionPolynomial^{- 1} (D_P)$
  and
  
  \begin{align*}
    \poincarefunctioninverse = & \poincarefunction_{|D_0}^{- 1} .
  \end{align*}
  
  Note that, for $w \notin D_P$, the sets $ P^{- n} (w) \cap A_P$ are empty for $n \leq \nu$,
  with some $\nu = \nu (w)$ such that
  $\nu \tendsto \infty$ as $w \tendsto \infty$. And more precisely $\nu (w)$ grows
  like $ \frac{ \log \log |w| }{ \log d }$.
  From Observation  \ref{obs:preimage structure}, it follows that,
  for any $ t > \hypdim P $ and $w$ with $|w|$ large,
  \begin{equation*}
  \mcylpartfun \left( t, \poincarefunction, w \right) =
     \familysum{n \geq \nu}{} \familysum{\zeta \in P^{- n} (w)
     \setintersection A_P}{} \frac{1}{\left| u \poincarefunction' (u)
     \right|^t} \frac{| w |^t}{\left| {P^{\iterated n}}' (\zeta) \right|^t},
  \end{equation*}
  where $u = \poincarefunctioninverse (\zeta)$.
  
  Since $A_0$ is fixed and $\poincarefunction$ is close to the identity near
  $0$, the factor $u \poincarefunction' (u)$ is bounded away from $0$ and
  $\infty$ uniformly in $u \in A_0$. It follows that it is enough to verify the following claim. 
  
  \begin{claim}
    \label{lem:partition functions tends to 0}Denote by
    $\euclideanpartitionfunction (\DimensionParameter,
    \PoincareFunctionPolynomial^{\Iterated{n}}, w)$ the first partition
    function for $\PoincareFunctionPolynomial^{\Iterated{n}}$ in the euclidian
    metric:
    \begin{equation}
      \euclideanpartitionfunction (\DimensionParameter,
      \PoincareFunctionPolynomial^{\Iterated{n}}, w) = \sum_{z \in
      \PoincareFunctionPolynomial^{- n} (w)}
      \frac{1}{ {}^{} \abs{ \p{ \PoincareFunctionPolynomial^{\Iterated{n}} }' (z) }^{\DimensionParameter}} . % the {}^{} because LATEX doesn't put the fraction bar correctly
    \end{equation}
    Then, the series $\familysum{n \geq \nu}{} |w|^{\DimensionParameter}
    \euclideanpartitionfunction (\DimensionParameter,
    \PoincareFunctionPolynomial^{\Iterated{n}}, w)$ converges to $0$ as $w
    \tendsto \infty$.
  \end{claim}
  
  \begin{subproof}[Proof of the Claim]
  The polynomial $\PoincareFunctionPolynomial$ is topologically Collet-Eckmann and  $t > \hyperbolicdimension P$.
   Hence the topological pressure associated to $\PoincareFunctionPolynomial$ is negative at exponent $t$.

   From that it follows that for any bounded set $F$ compactly embedded in the basin of attraction of $\infty$,
   there exists $C_F>0$ and $\pi_t>0$
   such that for all $\zeta\in F$ and all $m\geq 0$,
   \begin{equation}\label{eq:unif bound tx to tce}
     \sum_{z\in \PoincareFunctionPolynomial^{-m}(\zeta)} \abs{ { P^m }'(z) }^{ -t } \leq C_F e^{-\pi_t m}.
   \end{equation}

    Suppose $F$ is an annulus which is a
    foundamental domain for the dynamics of $\PoincareFunctionPolynomial$ in
    the basin of attraction of $\infty$.
    
    Given $ \eps > 0 $, we can choose $F$ close enough to $\infty$ so that for $w$ in the
    unbounded component $U$ of the complement of $F$, there exists
    $ k = k (w) \in \naturalnumbers $, with $ \nu (w) \geq k(w) \tendsto \infty $ as $ w \tendsto \infty $,
    and satisfying the following.
    The set $ P^{ - k - 1 } (w) $
    is a subset of $F$
    and for all $\zeta \in  \PoincareFunctionPolynomial^{-k} (w)$,
    we have
    $\left| \left( \PoincareFunctionPolynomial^{\Iterated{k}} \right)' (\zeta) \right| 
    \geq
    (1 - \eps)^{k+1} \PoincareFunctionPolynomialDegree^k \frac{|w|}{| \zeta |}$.
    Indeed,
    $P' (\xi) = \left(1 + \bigo \left( \frac{1}{\xi} \right) \right) d \xi^{d - 1}$
    for $\xi\in U\union F$,
    which implies
    \begin{eqnarray*}
      {P^k}' (\zeta) & = & P' (\zeta) P' (P (\zeta)) \cdots P' (P^{k - 1} (\zeta))\\
      & = &
      \left( 1 + \bigo \left( \frac{1}{\zeta} \right) \right)
      \left( 1 + \bigo \left( \frac{1}{\zeta^d} \right) \right)
      \cdots
      \left( 1 + \bigo  \left( \frac{1}{\zeta^{d^{k - 1}}} \right) \right)
      d^k \zeta^{d^k - 1} .
    \end{eqnarray*}
    
    Now we suppose $F$ fixed and consider $w\in U$,
    \begin{align*}    
     \familysum{n \geq \nu}{} |w|^{\DimensionParameter}
     \euclideanpartitionfunction (\DimensionParameter, \PoincareFunctionPolynomial^{\Iterated{n}}, w)
         & \leq 
         \familysum{n \geq \nu}{}
         \familysum{\zeta\in \PoincareFunctionPolynomial^{-k}(w)}{}
           \frac{|\zeta|^t}{\left((1-\eps)^t d^t\right)^k}
         \familysum{z\in \PoincareFunctionPolynomial^{-(n-k)}(\zeta)}{}
           \left|{\PoincareFunctionPolynomial^{n-k}}'(z)\right|^{-t} \\
     & \leq
         \frac{C_F}{\left((1-\eps)^t d^t\right)^k}
         \familysum{n \geq \nu}{}
           \,e^{-(n-k)\pi_t}
           \familysum{\zeta\in \PoincareFunctionPolynomial^{-k}(w)}{}
             |\zeta|^t ,
    \end{align*}
    where we have applied \eqref{eq:unif bound tx to tce} with $m=n-k\geq 0$.
    The set $F$ is bounded
    and the set $\PoincareFunctionPolynomial^{-k}(w)$ has at most $d^k$ points,
    thus,
    \begin{equation}
      \familysum{n \geq \nu}{}
         |w|^{\DimensionParameter}
         \euclideanpartitionfunction (\DimensionParameter, \PoincareFunctionPolynomial^{\Iterated{n}}, w)
       \leq 
       C_{F,t} \left(\frac{d}{(1-\eps)^t d^t}\right)^k 
    \end{equation}
    for some constant $C_{F,t}>0$ depending only on $F$ and $ t > \hypdim \PoincareFunctionPolynomial $.
    
   As $ \DimensionParameter > 1 $, it is possible to choose $\eps>0$
    so that the right hand side tends to $0$ as $w\tendsto\infty$.
\end{subproof}

{\  }
\end{proof}

\begin{proof}[Proof of Theorem \ref{thm:evhypdim of linearisers}]
 Suppose that $L$ is a Poincar\'e function for a topological Collet-Eckmann polynomial
  $P$ with connected Julia set.
Then $ \evhypdim L \geq \hypdim P $ by 
  Theorem \ref{thm:evhypdim ge hypdim} and 
  $\evhypdim L \leq \vanishingexponent(L) \leq \hypdim P < 2$ 
  by Lemma \ref{lem: evhypdim L leq hypdim P for TCE J connected P} and Theorem~\ref{thm:dim of julia of tce}.
  Recall that $ \evhypdim \geq 1 $, so 
  we have  
     \[ 1\leq \evhypdim(L) = \vanishingexponent(L) = \hypdim(P) < 2, \]
    establishing the first claim of the theorem. 
    If $L$ is of disjoint type, then $\hypdim(L)<2$ by 
    Corollary~\ref{cor:finitess exponent and hypdim}. 
     
     Moreover,
    if $P$ is not conformally conjugate to a Chebyshev polynomial or a power map, 
   then $\hyperbolicdimension P > 1 $ by 
   Theorem \ref{thm:dim of julia of tce}.
 \end{proof}

\section{Non invariance inside quasiconformal classes}

\begin{proof}[Proof of Corollary \ref{cor:noninvariance of evhypdim inside qc classes}]
  For $c$ in the main hyperbolic component $H$, let $\poincarefunction_c$ be the normalised
  Poincar{\'e} function of the unique repelling fixed point of the polynomial
  $P_c$. Then, by Proposition \ref{prop:sing set of linearisers}, the map $\poincarefunction_c$ has bounded postsingular set.
  Moreover its order $\rho$ is given by the following:
  \[
  \rho = \frac{\log 2}{\log | \lambda |}
  \]
  where $\lambda$ is the multiplier of the repelling fixed point
  (see {\cite{EremenkoSodin1990}}).
  
  Since the polynomial $P_c$ is hyperbolic, its hyperbolic dimension is equal
  to the dimension of its Julia set. Hence by Theorem \ref{thm:evhypdim of linearisers},
  \[
  \eventualhyperbolicdimension  \poincarefunction_c = \hausdorffdimension \juliaset (P_c) .
  \]
  The dimension of the Julia set is real analytic in $ c \in H $ and nonconstant, see
  {\cite{ruelle1982}}. Moreover, since the Julia set is a nondegenerate
  continuum, its dimension is bounded from below by $1$. Meanwhile, all
  nonzero $c$ inside this hyperbolic components are quasiconformally
  conjugated to each other. Then, from Remark \ref{rem:qc conj implies qc equiv of lin},
  it follows that the corresponding Poincar{\'e} functions are quasiconformally equivalent.

  Now we fix some $c\neq 0$ in the main hyperbolic component
  and we use the notation $\poincarefunction$ for $\poincarefunction_c$.
  It remains to show that for a choice of $\RescalingParameter>0$ the function
  $f=\RescalingParameter\poincarefunction$
  is quasiconformally conjugate to the map \adz{$ E (z) \defeq e^{ -2 } e^z $ (which is globally conjuated to $ z \mapsto e^z - 2 $)}
  on a neighbourhood of the Julia sets.
  
  Our aim is to apply Theorem 3.1 of \cite{Rempe2009}.
  We first need to recall a specific definition from this statement.
 If $f$ and $g$ are entire functions, they are said \textit{quasiconformaly equivalent near infinity}
  if there are quasiconformal functions $\ph,\psi:\complexnumbers\rightarrow\complexnumbers$
  such that $\ph\left(f(z)\right)=g\left(\psi(z)\right)$ whenever
  $|f(z)|$ and $|g(\psi(z))|$ are large enough.
  
  Then, by \cite{Rempe2009}, Theorem 3.1, if the functions $f$ and $g$ are of disjoint-type
  and quasiconformally equivalent near infinity,
  with the equivalence holding on disjoint-type tracts,
  then they are quasiconformally conjugated on a neighborhood of their respective Julia sets.

  We will show that $ \poincarefunction $ and $ E $ are quasiconformally equivalent near infinity.
  Before that we need to check that the map $ L $ has only one tract
  and that this tract is a quasicircle.
  
Since the closure of the critical orbit is a subset of the interior of the filled Julia set $ \filledjuliaset (P_c) $ of 
  the polynomial $P_c$,
  we can choose the tracts of $\poincarefunction$
  in a such way that
  their boundaries is the set $\Gamma=\poincarefunction^{-1}(\juliaset(P_c))$.
  Since the Julia set of $ P_c $ is a Jordan curve,
  the set $ \Gamma $ consists of simple (unbounded) curves.
  Moreover it follows from the invariance property of the Julia set 
  that $ \multiplier^n \Gamma = \Gamma $ for all $ n \in \integers $.
  
 Since $ \poincarefunction $ is univalent near $ 0 $
  and since the Julia set of $ P_c $ is a Jordan curve,
  the intersection of 
  $ \Gamma $ with a small disk $ D = \disk \p{ 0, r_0 } $ around $ 0 $ is made up of simple curves
  consisting of the points $ z \in D $ such that $ L (z) \in \juliaset (P_c) $.
 For $ r_0 > 0 $ small enough there is only one connected component $ \Gamma_0 $ of $ \Gamma $ that intersects $ D $.
  Moreover this component is invariant by multiplication by $ \multiplier $.
  This implies that there is only one tract.
  Indeed if $ \Gamma_1 $ is a connected component of $ \Gamma $
  then for any $ n \geq 0 $ large enough
  $ \multiplier^{ -n } \Gamma_1 $ intersects $ D $
  hence $ \Gamma_0 $ thus $ \Gamma_1 $ intersects $ \multiplier^n \Gamma_0 \subset \Gamma_0 $.

 The set $\juliaset(P_c)$ is a quasicircle.
 For any finite collection of points $ z_1, z_2, z_3, \dots $ in $ \Gamma $
 one can find $ n $  such that $ \multiplier^{ -n } z_1, \multiplier^{ -n } z_2, \multiplier^{ -n } z_3, \dots $ 
 all belong to $ D \intersection \Gamma $.
 If follows from that and from the geometric characterization of quasicircle
 (see e.g. \cite{Ahlfors1963})
 that the curve $ \Gamma \union \setof{ \infty } $ is also a quasicircle.
 As a consequence the map $\poincarefunction$
 has only one tract over $ \complexnumbers \setcomplement \filledjuliaset (P_c) $
 and it is a quasidisk.

  The Julia set $ \juliaset (P_c) $ is bounded and bounded away from $ 0 $,
  hence the preimage of $ \complexnumbers \setcomplement \filledjuliaset (P_c) $
  by the map $ E $
  is a (single) tract of $ E $ and a quasidisk.
  By Remark 2.7 of \cite{Rempe2009}, it follows that $\poincarefunction$
  is quasiconformally equivalent to $ E $ on a tract, i.e.
  there are quasiconformal functions $\ph,\psi:\complexnumbers\rightarrow\complexnumbers$
  such that $\ph \p{ E (z) } = \poincarefunction \p{ z } $
  whenever $ \realpart (z) $ is large enough.
\end{proof}

\section{A hyperbolic entire function with hyperbolic dimension \texorpdfstring{$ 2 $}{2}}

\begin{proof}[Proof of Corollary \ref{cor:exponentiallineariser}]
Let $f(z)=\twopii e^z$. The only finite singular value of $f$ is $0$ and it is sent onto the repelling fixed point $\twopii$.
 Moreover, the two-dimensional Lebesgue measure of the escaping set of $f$ is $0$ (see for example \cite{EremenkoLyubich1992}, Theorem 7).
 As a consequence, for almost all point $z$ in $\complexnumbers$, there exists $\delta>0$
 and a sequence of natural numbers $n_k\tendsto\infty$
 such that the spherical distance between $f^{n_k}(z)$
 and the postsingular set $\{0,\twopii,\infty\}$ of $f$
 is at least $\delta$\footnote{In other words, almost all points belong to the radial Julia set of $f$.}.
By the main result of \cite{Rempe2009a},
 it follows that the hyperbolic dimension of $f$ is $2$.
Hence, it follows from Theorem \ref{thm:Eventual hyperbolic dimension of linearisers II} that
 $\eventualhyperbolicdimension \poincarefunction=2$.

Now, let $\RescalingParameter=\poincarefunction'(0)$ and let $r=\frac{\pi}{8|\lambda|}$.
 We show that if $|\RescalingParameter|<\frac{1}{20}$
 then the image of the disk $\disk\left(0,r\right)$ by $\poincarefunction$
 is a compact subset of itself.
 
Indeed, the mapping $f$ is univalent on the disk $\disk(\twopii,\pi)$,
 and by the $1/4$ Theorem,
 $$f\left(\disk(\twopii,\pi)\right)\supset\disk\left(\twopii,\frac{2\pi\cdot\pi}{4}\right)\supset\disk(\twopii,\pi).$$
 Hence the lineariser $\poincarefunction$ has an inverse defined on that disk.
Hence by the $1/4$ Theorem, $\poincarefunction$ is univalent on
$\disk(0, \pi / (4|\lambda|))
 \subset
 \left(\poincarefunction_{\restricted \disk(\twopii,\pi)}\right)^{-1}\left(\disk\left(\twopii,\pi\right)\right)$.
By the Koebe Distortion Theorem, we have
$$\poincarefunction\left(\disk(0,r)\right)\subset\disk(\twopii,2r|\RescalingParameter|)=\disk(\twopii,\pi/2).$$
The claim follows then from the assumption $|\lambda|<1/20$.

According to Proposition \ref{prop:sing set of linearisers},
 the singular set of $\poincarefunction$ is the postsingular set of $f$.
Hence it consists only of the two points $0$ and $\twopii$.
As a consequence, the disk $\disk\left(0,r\right)$ is a subset of the Fatou set of $\poincarefunction$
 which absorbs all the singular orbits. Hence, the function $\poincarefunction$ is hyperbolic and its Fatou set is connected.
\end{proof}

\begin{appendices}

\section{Existence of an IFS}\label{sec:hypdim}

In this section we show that near every point of the Julia set of an entire
function $f$, there is a finite conformal iterated function system, made of
inverse branches of iterates of $f$, whose dimension is arbitrarily close to
the hyperbolic dimension of $f$. Then we use this to bound the critical
exponent of local pressures from below by the hyperbolic dimension.

This %will be used later on, 
is used
when estimating the eventual hyperbolic dimension
 of a Poincar\'e function from below.

\begin{lem}
  \label{lemma:existence-of-an-iterated-function-system}Let $f$ be an entire
  function and $U$ an open set intersecting the Julia set of $f$. Then, for
  any $\eps > 0$, there exists a finite conformal iterated function system
  defined on a subset of $U$ made of inverse branches of iterates of $f$, such
  that the dimension of its limit set is at least $\hyperbolicdimension f -
  \eps$.
\end{lem}

\begin{rem}
  This is also true in more generality for Ahlfors island maps.
\end{rem}

\begin{proof}
  Let $\ifsname = \left( \ifsdomain, \left( \ifsrange_i, \ph_i \right)_{i \in I} \right)$ \adz{be} a
  finite conformal iterated function whose limit set has dimension at least
  $\hyperbolicdimension f - \eps$, where, for all $i$, $\ifsrange_i \subset \ifsdomain$ and
  $\ph_i : \ifsdomain \rightarrow \ifsrange_i$ is an inverse branch of some iterate
  $f^{\iterated m_i}$ of $f$. 
  
  Since $\ifsdomain$ intersects the Julia set of $f$, there exists a natural number $N$
  such that $f^{\iterated N} (\ifsdomain) \intersection U$ has nonempty interior and intersects $ \juliaset (f) $.
  Let $V \subset \ifsdomain$ closed with nonempty interior such that
  $f^{\iterated N}_{\restricted V} : V \rightarrow
  U' \defeq f^{\iterated N} (V) \subset U$ is univalent and $U'$ intersects
  the Julia set of $f$.
  
  Since $U'$ intersects the Julia set of $f$, there exists a natural number
  $M$ such that $f^{\iterated M} (U') \supset \ifsdomain$. Then one can find, for all
  $i \in I$, closed sets with nonempty interiors $E_i \subset U'$ such that
  $f^{\iterated M}_{\restricted E_i} : E_i \rightarrow f^{\iterated M} (E_i) \subset \ifsrange_i$
  is univalent.
  
  Let $k$ be any sufficiently large positive integer and consider the
  conformal iterated function system $Y_k$ defined on $U$ by the family of
  mappings $f^{- M}_{\restricted E_{i_0}} \composedwith \ph_{i_k}
  \composedwith \ldots . \composedwith \ph_{i_1} \composedwith \ph_{i_0} \composedwith f^{- N}_{\restricted V}$,
  where $(i_{0,} i_1, \ldots ., i_k) \in I^{k + 1}$.
  We claim that the dimension of its limit set is at least
  $\hyperbolicdimension f - \eps$ \footnote{It is indeed the same as the
  original system $\ifsname$.}. Indeed, by using the pressure of the system $X$, we
  can find, for any $t < \hyperbolicdimension f - \eps$ and for any constant
  $C > 0$, a $k_0$ such that for all $k \geq k_0$, we have
  \[ \familysum{i \in I^k}{} \left\| \left( \ph_{i_k} \composedwith \ldots .
     \composedwith \ph_{i_1} \right)' \right\|^t > C. \]
  By a standard distortion argument, one can find $k$ such that the first
  partition function of the system $Y_k$ is greater than a given constant. It
  follows that the pressure of the system $Y_k$ is positive for the exponent
  $t$.
\end{proof}

Lemma \ref{lem:lower exp growth} is a corollary of Lemma \ref{lemma:existence-of-an-iterated-function-system}.

\begin{proof}[Proof of Lemma \ref{lem:lower exp growth}]
  Let $\ifsname = \left(\ifsdomain, \left( \ifsrange_i, \ph_i \right)_{1 \leq i \leq I} \right)$ be
  the iterated function system from Lemma
  \ref{lemma:existence-of-an-iterated-function-system}, with
  $\ph_i = \left(f^{\iterated m_i}_{\restricted \ifsrange_i} \right)^{- 1}$.
  Denote $\topologicalpressure_\ifsname$ its pressure.
  By assumption,
  $\topologicalpressure_\ifsname (t) > 0$.
  Let $\eps > 0$ be such that $\beta_0 = \topologicalpressure_\ifsname (t) - \eps > 0$.
  Then, for all $p$ large enough,
 \begin{equation*}
  \familysum{| \alpha | = p}{} \norm{\ph_{\alpha_1}'}^t \cdot \ldots \cdot \norm{\ph_{\alpha_p}'}^t \geq e^{\beta_0 p}.
 \end{equation*}
  Introduce, for $ i = 1, \dots, I $,
  $\lambda_i = \left\| \ph_i' \right\|^t$ and
  $\Lambda = \familysum{i = 1}{I} \lambda_i$.
  Note that
  \begin{equation}
   \Lambda^p = \familysum{ \abs{ \alpha } = p }{} \lambda_{ \alpha_1 } \dots \lambda_{ \alpha_p }
   = \familysum{ \abs{ \alpha } = p}{} \norm{ \ph_{\alpha_1}' }^t \cdot \ldots . \cdot \norm{ \ph_{\alpha_p}' }^t .
  \end{equation}

  Define $\eps_i = \lambda_i / \Lambda$ and
  for each $p$ choose a finite sequence of natural numbers $k_1 (p), \ldots .,k_I (p)$
  such that
  \begin{equation}
    \eps_i p - 1 \leq k_i (p) < \eps_i p + 1 \label{eq:bounds on the ki}
  \end{equation}
  and
  \[ \familysum{i = 1}{I} k_i (p) = p. \]
  We will show that if
  \begin{equation}
    \nu_p = \familysum{i = 1}{I} k_i (p) m_i, \label{eq:definition of np}
  \end{equation}
  then the properties of the statement are satisfied for any $p$ large.
  
  Let's fix $p$ and write $k_i$ instead of $k_i (p)$. Consider the set $S_p$
  of multi-indexes $\alpha$ of length $p$ such that the number of $0 \leq j \leq p$
  such that $\alpha_j = i$ is precisely $k_i$ for all $1 \leq i \leq I$.
  Then,
  \[ \frac{1}{\Lambda^p} \familysum{\alpha \in S_p}{} \lambda_{\alpha_1} \cdot \ldots . \cdot \lambda_{\alpha_p}
  = \frac{p!}{k_1 ! \cdots k_I !} \cdot \frac{\lambda_1^{k_1} \cdot \ldots \cdot \lambda_I^{k_I}}{\Lambda^p} .
  \]
  Note that for $p \geq 1$ large enough, $k_i > 0$. From Stirling's
  approximation, it follows that for any $p$ large enough,
  \[ \frac{p!}{k_1 ! \cdots k_I !} \geq \left( \frac{1}{2 \pi}
     \right)^{\frac{I - 1}{2}} \cdot \sqrt{\frac{p}{k_1 \cdot \ldots . \cdot
     k_I}} \cdot \frac{p^p}{k_1^{k_1} \cdot \ldots . \cdot k_I^{k_I}} . \]
  From equation {\eqref{eq:bounds on the ki}}, it follows that, for $p$ large,
  \[ k_1 \cdot \ldots . \cdot k_I \leq 2 p^I  \eps_1 \cdot \ldots . \cdot
     \eps_I \]
  and
  \begin{align}
        k_1^{k_1} \cdot \ldots . \cdot k_I^{k_I}
        & \leq \p{ \familyproduct{ i }{} \p{ \eps_i p }^{ k_i } } \familyproduct{ i }{} \p{ 1 + \frac{ 1 }{ \eps_i p }}^{ k_i } \\
        & \leq 
        C_1 p^{ p } \p{
           \eps_1^{k_1} \cdot \ldots . \cdot \eps_I^{k_I} }
       \end{align}
  with some $C_1 > 0$ depending only on $\left( \eps_i \right)_{i = 1, \ldots, I}$.
  Hence we have found a constant $C_2 > 0$ which depends only on
 $(I, \lambda_1, \ldots ., \lambda_I)$ and which is such that
  \[ \frac{p!}{k_1 ! \cdots k_I !} \cdot \frac{\lambda_1^{k_1} \cdot \ldots .
     \cdot \lambda_I^{k_I}}{\Lambda^p} \geq C_2 \cdot  \frac{1}{p^{\delta}}
     \cdot \frac{1}{\Lambda^p} \cdot \familyproduct{i = 1}{I} \left(
     \frac{\lambda_i}{\eps_i} \right)^{k_i} = \frac{C_2}{p^{\delta}}.
  \]
  with $\delta = (3 I - 1) / 2$.
  
  This implies that
  \begin{align}
   \familysum{\alpha \in S_p}{} \left\| \ph_{\alpha_1}' \right\|^t \cdot
     \ldots . \cdot \left\| \ph_{\alpha_p}' \right\|^t
     & = 
     \frac{ 1 }{ \Lambda^p }
     \p{ \familysum{\alpha \in S_p }{ } \lambda_{\alpha_1} \cdot  \ldots \cdot \lambda_{\alpha_p}  }
     \p{ \familysum{| \alpha | = p}{} \norm{ \ph_{\alpha_1}' }^t \cdot \ldots . \cdot \norm{ \ph_{\alpha_p}' }^t } \\
     & \geq
     \frac{C_2}{p^{\delta}} \cdot \familysum{| \alpha | = p}{} \left\|
     \ph_{\alpha_1}' \right\|^t \cdot \ldots . \cdot \left\| \ph_{\alpha_p}'
     \right\|^t \geq \frac{C_2 e^{\beta_0 p}}{p^{\delta}}.
  \end{align} 
   Then, by standard distortion estimates, and
  from the definition {\eqref{eq:definition of np}} of $\nu_p$, there exists
  $C > 0$ such that for $w \in \ifsdomain$,
  \begin{align}
   \familysum{z \in f^{- \nu_p} (w) \setintersection \ifsdomain}{} \abs{ {f^{\iterated \nu_p}}' (z) }^{- t}
   & \geq C  \familysum{\alpha \in S_p}{} \norm{ \ph_{\alpha_1}' }^t \cdot \ldots . \cdot \norm{ \ph_{\alpha_p}' }^t 
    \geq \frac{C'}{\nu_p^{\delta}}
     e^{\beta \nu_p},
  \end{align}
  with $\beta_1 = \beta_0 / (2 a) $  and $a = \familysum{i = 1}{I} \eps_i m_i$.

  For any $ p \geq 1 $,
  choose a bijection $ \chi : \setof{ 1, \dots, I_p } \to S_p $ with $ I_p = \numberofelements S_p $,
  and define, for $ 1 \leq i \leq I_p $,
  $ D_i^p \defeq \ph_{ \chi (i) } (D) $.
  We will have equation \eqref{pressure-explosion-on-the-ifs}
  for $ \beta > 0 $ smaller than $ \beta_1 $ and an appropriate choice of $ C > 0 $.
  Finally, note that $a p - b \leq \nu_p < a p + b$ for $b = \familysum{i = 1}{I} m_i$.
\end{proof}

\end{appendices}

\bibliography{all}

\providecommand{\bysame}{\leavevmode\hbox to3em{\hrulefill}\thinspace}
\providecommand{\MR}{\relax\ifhmode\unskip\space\fi MR }
% \MRhref is called by the amsart/book/proc definition of \MR.
\providecommand{\MRhref}[2]{%
  \href{http://www.ams.org/mathscinet-getitem?mr=#1}{#2}
}
\providecommand{\href}[2]{#2}
\begin{thebibliography}{PRLS04}

\bibitem[Ahl63]{Ahlfors1963}
Lars~V. Ahlfors, \emph{Quasiconformal reflections}, Acta Math. \textbf{109}
  (1963), 291--301. \MR{0154978}

\bibitem[AL15]{Man-AvilaLyubich2015}
Artur Avila and Mikhail Lyubich, \emph{{L}ebesgue measure of {F}eigenbaum
  {J}ulia sets}, preprint arXiv:1504.02986, 2015.

\bibitem[Bea91]{Book-Beardon1991}
Alan~F. Beardon, \emph{Iteration of rational functions}, Graduate Texts in
  Mathematics, vol. 132, Springer-Verlag, New York, 1991, Complex analytic
  dynamical systems.

\bibitem[BK07]{BaranskiKarpinska2007}
Krzysztof Bara{\'n}ski and Bogus{\l}awa Karpi{\'n}ska, \emph{Coding trees and
  boundaries of attracting basins for some entire maps}, Nonlinearity
  \textbf{20} (2007), no.~2, 391--415. \MR{2290468 (2007m:37105)}

\bibitem[BKZ09]{BaranskiKarpinskaZdunik2009}
Krzysztof Bara{\'n}ski, Bogus{\l}awa Karpi{\'n}ska, and Anna Zdunik,
  \emph{Hyperbolic dimension of {J}ulia sets of meromorphic maps with
  logarithmic tracts}, Int. Math. Res. Not. IMRN (2009), no.~4, 615--624.
  \MR{2480096 (2009k:37109)}

\bibitem[BKZ12]{BaranskiKarpinskaZdunik2012}
\bysame, \emph{Bowen's formula for meromorphic functions}, Ergodic Theory
  Dynam. Systems \textbf{32} (2012), no.~4, 1165--1189. \MR{2955309}

\bibitem[Bro65]{Brolin1965}
Hans Brolin, \emph{Invariant sets under iteration of rational functions}, Ark.
  Mat. \textbf{6} (1965), 103--144 (1965). \MR{0194595 (33 \#2805)}

\bibitem[DU91]{DenkerUrbanski1991d}
Manfred Denker and Mariusz Urba{\'n}ski, \emph{On {S}ullivan's conformal
  measures for rational maps of the {R}iemann sphere}, Nonlinearity \textbf{4}
  (1991), no.~2, 365--384. \MR{1107011 (92f:58097)}

\bibitem[EL92]{EremenkoLyubich1992}
Alexandre~{\`E}. Eremenko and Mikail~Yu Lyubich, \emph{Dynamical properties of
  some classes of entire functions}, Ann. Inst. Fourier (Grenoble) \textbf{42}
  (1992), no.~4, 989--1020. \MR{1196102 (93k:30034)}

\bibitem[ER15]{EpsteinRempeGillen2015}
Adam Epstein and Lasse Rempe{-Gillen}, \emph{{On invariance of order and the
  area property for finite-type entire functions.}}, {Ann. Acad. Sci. Fenn.,
  Math.} \textbf{40} (2015), no.~2, 573--599 (English).

\bibitem[Ere89]{Eremenko1989}
Alexandre~{\`E}. Eremenko, \emph{On the iteration of entire functions},
  Dynamical systems and ergodic theory ({W}arsaw, 1986), Banach Center Publ.,
  vol.~23, PWN, Warsaw, 1989, pp.~339--345. \MR{1102727 (92c:30027)}

\bibitem[ES90]{EremenkoSodin1990}
A.~{\`E}. Eremenko and M.~L. Sodin, \emph{Iterations of rational functions and
  the distribution of the values of {P}oincar\'e functions}, Teor. Funktsi\u\i\
  Funktsional. Anal. i Prilozhen. (1990), no.~53, 18--25. \MR{1077218
  (92d:30016)}

\bibitem[GM05]{GarnettMarshall2005}
John~B. Garnett and Donald~E. Marshall, \emph{Harmonic measure}, New
  Mathematical Monographs, vol.~2, Cambridge University Press, Cambridge, 2005.

\bibitem[GV73]{GehringVaisala1973}
F.~W. Gehring and J.~V{\"a}is{\"a}l{\"a}, \emph{Hausdorff dimension and
  quasiconformal mappings}, J. London Math. Soc. (2) \textbf{6} (1973),
  504--512. \MR{0324028 (48 \#2380)}

\bibitem[KU08]{KotusUrbanski2008}
Janina Kotus and Mariusz Urba{\'n}ski, \emph{Fractal measures and ergodic
  theory of transcendental meromorphic functions}, Transcendental dynamics and
  complex analysis, London Math. Soc. Lecture Note Ser., vol. 348, Cambridge
  Univ. Press, Cambridge, 2008, pp.~251--316.

\bibitem[Lyu86]{Lyubich1986}
M.~Yu. Lyubich, \emph{Generic behavior of trajectories of the exponential
  function}, Uspekhi Mat. Nauk \textbf{41} (1986), no.~2(248), 199--200.

\bibitem[Lyu87]{Lyubich1987}
\bysame, \emph{The measurable dynamics of the exponential}, Sibirsk. Mat. Zh.
  \textbf{28} (1987), no.~5, 111--127. \MR{924986 (89d:58071)}

\bibitem[Lyu00]{Lyubich2000}
Mikhail Lyubich, \emph{The quadratic family as a qualitatively solvable model
  of chaos}, Notices Amer. Math. Soc. \textbf{47} (2000), no.~9, 1042--1052.

\bibitem[Mak85]{Makarov1985}
N.~G. Makarov, \emph{On the distortion of boundary sets under conformal
  mappings}, Proc. London Math. Soc. (3) \textbf{51} (1985), no.~2, 369--384.

\bibitem[May17]{Mayer17}
Volker Mayer, \emph{A lower bound of the hyperbolic dimension for meromorphic
  functions having a logarithmic hölder tract}, Preprint arXiv:1709.02188,
  2017.

\bibitem[MBP12]{MihaljevicBrandtPeter2012}
Helena Mihaljevi{\'c}-Brandt and J{\"o}rn Peter, \emph{Poincar\'e functions
  with spiders' webs}, Proc. Amer. Math. Soc. \textbf{140} (2012), no.~9,
  3193--3205. \MR{2917092}

\bibitem[McM87]{McMullen1987}
Curtis~T. McMullen, \emph{Area and {H}ausdorff dimension of {J}ulia sets of
  entire functions}, Trans. Amer. Math. Soc. \textbf{300} (1987), no.~1,
  329--342. \MR{871679 (88a:30057)}

\bibitem[McM00]{McMullen2000}
\bysame, \emph{Hausdorff dimension and conformal dynamics. {II}.
  {G}eometrically finite rational maps}, Comment. Math. Helv. \textbf{75}
  (2000), no.~4, 535--593.

\bibitem[MU08]{MayerUrbanski2008}
Volker Mayer and Mariusz Urba{\'n}ski, \emph{Geometric thermodynamic formalism
  and real analyticity for meromorphic functions of finite order}, Ergodic
  Theory and Dynamical Systems \textbf{28} (2008), 915--946.

\bibitem[MU10]{MayerUrbanski2010}
\bysame, \emph{Thermodynamical formalism and multifractal analysis for
  meromorphic functions of finite order}, Mem. Amer. Math. Soc. \textbf{203}
  (2010), no.~954, vi+107. \MR{2590263}

\bibitem[MU19]{MayerUrbanski2019}
Volker Mayer and Mariusz Urbański, \emph{Thermodynamic formalism and integral
  means spectrum of asymptotic tracts for transcendental entire functions},
  Preprint arXiv:1709.05166, 2019.

\bibitem[MZ19]{MayerZdunik2019}
Volker Mayer and Anna Zdunik, \emph{The failure of ruelle's property for entire
  functions}, 2019.

\bibitem[PRL07]{PrzytyckiRiveraLetelier2007}
Feliks Przytycki and Juan Rivera-Letelier, \emph{Statistical properties of
  topological {C}ollet-{E}ckmann maps}, Ann. Sci. \'Ecole Norm. Sup. (4)
  \textbf{40} (2007), no.~1, 135--178. \MR{2332354 (2008j:37093)}

\bibitem[PRLS03]{PrzytyckiRiveraLetelierSmirnov2003}
Feliks Przytycki, Juan Rivera-Letelier, and Stanislav Smirnov,
  \emph{Equivalence and topological invariance of conditions for non-uniform
  hyperbolicity in the iteration of rational maps}, Inventiones mathematicae
  \textbf{151} (2003), no.~1, 29--63 (English).

\bibitem[PRLS04]{Przytycki2004}
Feliks Przytycki, Juan Rivera-Letelier, and Stanislav Smirnov, \emph{Equality
  of pressures for rational functions}, Ergod. Th. Dynam. Sys. \textbf{24}
  (2004), no.~3, 891--914.

\bibitem[Prz93]{Przytycki1993}
Feliks Przytycki, \emph{Lyapunov characteristic exponents are nonnegative},
  Proceedings of the American Mathematical Society \textbf{119} (1993), no.~1,
  pp. 309--317 (English).

\bibitem[Prz98]{Przytycki1998}
\bysame, \emph{Iterations of holomorphic {C}ollet-{E}ckmann maps: conformal and
  invariant measures. {A}ppendix: on non-renormalizable quadratic polynomials},
  Trans. Amer. Math. Soc. \textbf{350} (1998), no.~2, 717--742. \MR{1407501
  (98d:58155)}

\bibitem[Prz99]{Przytycki1999}
\bysame, \emph{Conical limit set and {P}oincar\'e exponent for iterations of
  rational functions}, Trans. Amer. Math. Soc. \textbf{351} (1999), no.~5,
  2081--2099. \MR{1615954 (99h:58110)}

\bibitem[PU10]{PrzytyckiUrbanski2010}
Feliks Przytycki and Mariusz Urbanski, \emph{Conformal fractals: Ergodic theory
  methods}, 1st ed., Cambridge University Press, New York, NY, USA, 2010.

\bibitem[Ree86]{Rees1986}
Mary Rees, \emph{The exponential map is not recurrent}, Mathematische
  Zeitschrift \textbf{191} (1986), no.~4, 593--598 (English).

\bibitem[Rem09a]{Rempe2009a}
Lasse Rempe, \emph{Hyperbolic dimension and radial {J}ulia sets of
  transcendental functions}, Proc. Amer. Math. Soc. \textbf{137} (2009), no.~4,
  1411--1420. \MR{2465667 (2010h:37100)}

\bibitem[Rem09b]{Rempe2009}
\bysame, \emph{Rigidity of escaping dynamics for transcendental entire
  functions}, Acta Math. \textbf{203} (2009), no.~2, 235--267. \MR{2570071
  (2011b:37084)}

\bibitem[Rem14]{RempeGillen2014}
Lasse Rempe{-Gillen}, \emph{Hyperbolic entire functions with full hyperbolic
  dimension and approximation by {E}remenko-{L}yubich functions}, Proceedings
  of the London Mathematical Society \textbf{108} (2014), no.~5, 1193--1225.

\bibitem[RG16]{RempeGillen2018}
Lasse Rempe-Gillen, \emph{Arc-like continua, julia sets of entire functions,
  and eremenko's conjecture}, 2016.

\bibitem[RS10]{RempeStallard2010}
Lasse Rempe and Gwyneth~M. Stallard, \emph{Hausdorff dimensions of escaping
  sets of transcendental entire functions}, Proc. Amer. Math. Soc. \textbf{138}
  (2010), no.~5, 1657--1665. \MR{2587450 (2011a:37097)}

\bibitem[RS17]{Rempe-GillenSixsmith2017}
Lasse Rempe{-Gillen} and Dave Sixsmith, \emph{Hyperbolic entire functions and
  the {E}remenko-{L}yubich class: class {$ {\mathcal B} $} or not class {$
  {\mathcal B} $}?}, Math. Z. \textbf{286} (2017), no.~3-4, 783--800.
  \MR{3671560}

\bibitem[Rue82]{ruelle1982}
David Ruelle, \emph{Repellers for real analytic maps}, Ergod. Theor. Dynam.
  Syst. \textbf{2} (1982), 99--107.

\bibitem[Shi98]{Shishikura1998}
Mitsuhiro Shishikura, \emph{The hausdorff dimension of the boundary of the
  mandelbrot set and julia sets}, Annals of Mathematics \textbf{147} (1998),
  no.~2, pp. 225--267.

\bibitem[Six14]{sixsmithclassB}
D.~J. Sixsmith, \emph{A new characterisation of the {E}remenko-{L}yubich
  class}, J. Anal. Math. \textbf{123} (2014), 95--105. \MR{3233575}

\bibitem[Sta99]{Stallard1999}
Gwyneth~M. Stallard, \emph{The {H}ausdorff dimension of {J}ulia sets of
  hyperbolic meromorphic functions}, Math. Proc. Cambridge Philos. Soc.
  \textbf{127} (1999), no.~2, 271--288.

\bibitem[Sul82]{Sullivan1982}
Dennis Sullivan, \emph{Discrete conformal groups and measurable dynamics},
  Bull. Amer. Math. Soc. (N.S.) \textbf{6} (1982), no.~1, 57--73. \MR{634434
  (83c:58066)}

\bibitem[Sul83]{Sullivan1983}
\bysame, \emph{Conformal dynamical systems}, Geometric Dynamics (Jr. Palis, J.,
  ed.), Lecture Notes in Mathematics, vol. 1007, Springer Berlin Heidelberg,
  1983, pp.~725--752 (English).

\bibitem[UZ03]{UrbanskiZdunik2003}
Mariusz Urba{\'n}ski and Anna Zdunik, \emph{The finer geometry and dynamics of
  the hyperbolic exponential family}, Michigan Math. J. \textbf{51} (2003),
  no.~2, 227--250. \MR{1992945 (2004d:37068)}

\bibitem[Zdu90]{Zdunik1990}
Anna Zdunik, \emph{Parabolic orbifolds and the dimension of the maximal measure
  for rational maps}, Invent. Math. \textbf{99} (1990), no.~3, 627--649.
  \MR{1032883 (90m:58120)}

\end{thebibliography}
\bibliographystyle{amsalpha}

\end{document}